\documentclass[10pt,a4paper]{amsart}
\usepackage[T1]{fontenc}
\usepackage{lmodern,amsmath,amssymb,amsthm}
\usepackage[a4paper,textwidth=16.1cm,textheight=23cm,centering]{geometry}
\usepackage{enumitem,longtable,hyperref}
\hypersetup{hidelinks,pdftitle={On the number of zero-sum subsequences in cyclic groups}}
\theoremstyle{plain}
\newtheorem{theorem}{Theorem}[section]
\newtheorem{proposition}[theorem]{Proposition}
\newtheorem{lemma}[theorem]{Lemma}

\newtheorem{conjecture}[theorem]{Conjecture}
\theoremstyle{definition}

\newcommand{\N}{\mathbb N}
\newcommand{\Z}{\mathbb Z}

\newcommand{\Fcal}{\mathcal F}

\DeclareMathOperator{\supp}{supp}
\numberwithin{equation}{section}
\setlist[enumerate]{label=\textup{(\arabic*)},leftmargin=*}
\title[Zero-sum subsequence counts]{Enumeration of zero-sum subsequences: A proof of Bialostocki-Lotspeich conjecture}
\author{Guoqing Wang}
\address{School of Mathematical Sciences \\ Tiangong University \\ Tianjin 300387, P.R. China}
\email{gqwang1979@aliyun.com}
\author{Xiangneng Zeng}
\address{Sino-French Institute of Nuclear Engineering and Technology,
Sun Yat-Sen University, Guangzhou 510275, P.R. China}
\email{zengxneng@mail.sysu.edu.cn}
\date{}
\keywords{Bialostocki-Lotspeich conjecture; Zero-sum sequences; Enumeration of subsequences; Erd\H{o}s--Ginzburg--Ziv theorem}
\begin{document}
\begin{abstract}
Let $C_n$ be a cyclic group of order $n\ge 2$, and let
$S=g_1\cdot\ldots\cdot g_{\ell}$ be a sequence over $C_n$, with repetitions
allowed. We denote by $\mathsf N_0^n(S)$ the number of subsets
$I\subseteq[1,\ell]$ with $|I|=n$ and $\sum_{i\in I}g_i=0$.
A conjecture proposed by Bialostocki and Lotspeich \cite[Conjecture 5.4]{BialostockiLotspeich1992}
in 1992 states that
$$\mathsf N_0^n(S)\ge
\binom{\lfloor |S|/2\rfloor}{n}+
\binom{\lceil |S|/2\rceil}{n}.$$
Several partial results toward this conjecture are known. In this paper, we prove this conjecture completely.
\end{abstract}

\maketitle

\section{Introduction}
\label{sec:introduction}

Let $G$ be an additive finite abelian group, and let $\mathcal F(G)$
denote the free abelian monoid with basis $G$. Its elements are called
\emph{sequences} over $G$. A sequence has repetitions but no prescribed
ordering. For $S=g_1\cdot\ldots\cdot g_{\ell}\in\mathcal F(G)$,
we write $|S|=\ell$ and $\sigma(S)=g_1+\cdots+g_{\ell}$.
The classical theorem of Erd\H{o}s, Ginzburg, and Ziv
\cite{ErdosGinzburgZiv1961} states that every sequence of length $2n-1$
over a cyclic group $C_n$ contains a zero-sum subsequence of length $n$.

For longer sequences, the corresponding counting problem asks how many
such subsequences must occur. Following
\cite[Definition~8.1]{GaoGeroldinger2006}, we set
\begin{equation}\label{eq:intro-count}
 \mathsf N_0^n(S)=
 \left|\left\{I\subseteq[1,|S|]:
 |I|=n,\ \sum_{i\in I}g_i=0\right\}\right|.
\end{equation}
Thus subsequences are counted with the multiplicity of their appearance
in $S$.

Bialostocki \cite{Bialostocki1989} posed the following conjecture in a research proposal
from 1989; its published formulation is in his joint paper with M. Lotspeich
\cite[Conjecture~5.4]{BialostockiLotspeich1992}.
Caro restates this conjecture in
\cite[Conjecture~2.3]{Caro1996}. Gao and Geroldinger \cite[Conjecture~8.8]{GaoGeroldinger2006} attributed
the conjecture to Bialostocki and Lotspeich and stated it as follows.

\begin{conjecture}[Bialostocki--Lotspeich]
\label{con:BiaLot}
 Let $G$ be cyclic group of order $n\ge2$, and let $S\in\mathcal F(G)$.
 Then $$\mathsf N_0^n(S)\ge \binom{\lfloor |S|/2\rfloor}{n}
  +\binom{\lceil |S|/2\rceil}{n}.$$
\end{conjecture}

Several substantial cases were established during the 1990s.
F\"uredi and Kleitman \cite{FurediKleitman1993} proved the conjecture
for products of two distinct primes and for sufficiently long sequences.
The explicit long-sequence range recorded in
\cite[Theorem~8.9]{GaoGeroldinger2006} is $|S|\ge n^{6n}$.
Kisin \cite[Theorem~8]{Kisin1994} established the conjecture for
$n=p^a$ and $n=p^a q$, where $p,q$ are distinct primes and $a\ge1$,
without a restriction on $|S|$.
Gao \cite{Gao1999} treated the range $|S|<5n/2$ and obtained
structural information when there are few zero-sum subsequences of length
$n$. In particular, when $|S|=2n-1$, either
$\mathsf N_0^n(S)\ge n+1$ or
$S=a^n b^{n-1}$ with $\operatorname{ord}(a-b)=n$.

Grynkiewicz \cite{Grynkiewicz2006} made substantial progress
toward the conjecture. He proved that, for every finite
abelian group $G$ of order $n\ge30$ and every sequence
$S\in\mathcal F(G)$ with $|S|\ge2n-1$,
\begin{equation}\label{eq:intro-grynkiewicz}
 \mathsf N_0^n(S)\ge
\min\left\{
\binom{\left\lfloor |S|/2\right\rfloor}{n}
+\binom{\left\lceil |S|/2\right\rceil}{n}, \
\binom{|S|-n}{\left\lceil\frac{2n-1}{3}\right\rceil}
\right\}.
\end{equation}
More recently, M\'oricz and Nagy
\cite[Conjecture~1.4]{MoriczNagy2026} also stated this
conjecture in their study of extremal zero-sum counts.
Their main problem asks for the maximum number of zero-sum
subsequences in sequences of nonzero integers, whereas the
present paper determines the minimum number of zero-sum
subsequences of length $n$ in sequences over $C_n$.

In this paper, we prove Conjecture \ref{con:BiaLot} completely, and our result is as follows.

\begin{theorem}\label{thm:main}
 Let $n\ge2$, let $G$ be a cyclic group of order $n$, and let
 $S\in\mathcal F(G)$. Then
 \[
  \mathsf N_0^n(S)\ge
  \binom{\lfloor |S|/2\rfloor}{n}
  +\binom{\lceil |S|/2\rceil}{n}.
 \]
\end{theorem}

The proof proceeds by strong induction on $n$. After recalling
Kisin's theorem and Grynkiewicz's lower bound, we establish
binomial estimates for short sequences and reductions for even
moduli, large residue classes, and moduli divisible by $3$.
The remaining case is handled by a pairing argument using
Kneser's theorem and a representation-counting lemma. The final
section assembles these ingredients to complete the proof.

\section{Preliminaries}
\label{sec:preliminaries}

We follow the notation of \cite{GaoGeroldinger2006}. We denote by
$\N$ the set of positive integers and put $\N_0=\N\cup\{0\}$.
For real numbers $a\le b$, we write
$[a,b]=\{z\in\Z:a\le z\le b\}$. All groups are written additively.
The cyclic group of order $n$ is denoted by $C_n$.

Let $G$ be a finite abelian group. We set $G^\bullet=G\setminus\{0\}$
and denote by $\mathcal F(G)$ the free abelian monoid with basis $G$.
Its elements are called \emph{sequences over $G$}. A sequence can be
written in the form
\[
 S=g_1\cdot\ldots\cdot g_\ell
   =\prod_{g\in G}g^{\mathsf v_g(S)}\in\mathcal F(G).
\]
Here $\mathsf v_g(S)\in\N_0$ is the multiplicity of $g$ in $S$,
$|S|=\ell$ is the length of $S$, and
\[
 \sigma(S)=\sum_{i=1}^\ell g_i,\qquad
 \supp(S)=\{g\in G:\mathsf v_g(S)>0\},\qquad
 \mathsf h(S)=\max\{\mathsf v_g(S):g\in G\}.
\]
We put $\mathsf h(1)=0$ for the empty sequence $1$. A sequence $T$
is a subsequence of $S$ if $T\mid S$ in $\mathcal F(G)$. It is a
\emph{zero-sum sequence} if $\sigma(T)=0$, and $S$ is
\emph{zero-sum free} if it has no nonempty zero-sum subsequence.
We use
\[
 \Sigma(S)=\{\sigma(T):1\ne T\mid S\},\qquad
 \Sigma^*(S)=\Sigma(S)\cup\{0\}.
\]
For a map $\varphi:G\to H$, we write
$\varphi(S)=\varphi(g_1)\cdot\ldots\cdot\varphi(g_\ell)$.
In particular, $-S=(-g_1)\cdot\ldots\cdot(-g_\ell)$.

For $g\in G$ and $k\in\N_0$, we define
\begin{equation}
 \mathsf N_g^k(S)=
 \left|\left\{I\subseteq[1,\ell]: |I|=k,
                         \ \sum_{i\in I}g_i=g\right\}\right|,
 \qquad
 \mathsf N_g(S)=\sum_{k\ge0}\mathsf N_g^k(S).
 \label{eq:positional-count}
\end{equation}
Thus subsequences are counted with the multiplicity of their
appearance in $S$, as in \cite[Definition 8.1]{GaoGeroldinger2006}.
The empty subsequence is included, so $\mathsf N_0^0(S)=1$.
Whenever we partition or select terms of a sequence, we mean its
indexed positions, even if several terms are equal.
We denote
\begin{equation}\label{eq:intro-B}
 B(\ell,n)=\binom{\lfloor\ell/2\rfloor}{n}
             +\binom{\lceil\ell/2\rceil}{n}.
\end{equation}
The assertion that $\mathsf N_0^n(S)\ge B(|S|,n)$ for every
$S\in\mathcal F(C_n)$ is denoted by $\mathcal C(n)$.
This notation is an assertion about a group order, not about a
particular sequence length.

If $G=C_n$ and $q\mid n$, then $qG=\{qg:g\in G\}$ is the subgroup
of index $q$. We denote the quotient map by $\pi_q:G\to G/qG$.
After choosing a generator, its fibers are precisely the residue
classes modulo $q$. An element of $C_n$ fails to generate $C_n$ if
and only if it lies in $qG$ for some prime $q\mid n$.

We recall some known counting results used in the proof.

\begin{theorem}[Kisin \cite{Kisin1994}]
\label{thm:kisin}
The assertion $\mathcal C(n)$ holds if $n$ is a prime power or
$n=p^a q$, where $p,q$ are distinct primes and $a\in\N$.
\end{theorem}

\begin{theorem}[Grynkiewicz \cite{Grynkiewicz2006}]
\label{thm:grynkiewicz}
Let $G$ be a finite abelian group of order $n\ge30$ and let
$S\in\mathcal F(G)$ have length $\ell\ge n$. Then
\[
 \mathsf N_0^n(S)\ge
 \min\left\{B(\ell,n),
             \binom{\ell-n}{\lceil(2n-1)/3\rceil}\right\}.
\]
In particular, $\mathsf N_0^n(S)\ge B(\ell,n)$ if
$\ell\le19n/3$.
\end{theorem}

We shall also use Kneser's theorem in the following form. For
nonempty finite subsets $A_1,\ldots,A_r$ of an abelian group, let
$H=\{h:(A_1+\cdots+A_r)+h=A_1+\cdots+A_r\}$ be the stabilizer
of their sumset. Then
\begin{equation}
 |A_1+\cdots+A_r|
 \ge \sum_{i=1}^r|A_i+H|-(r-1)|H|.
 \label{eq:kneser}
\end{equation}
See \cite{Kneser1953} or the formulation in
\cite{Grynkiewicz2006}. In the quotient by $H$ the sumset has
trivial stabilizer.

We finish this section with two elementary facts about positional
counts. The second is the source of the extra representations used
in the final argument.

\begin{lemma}
\label{lem:majority}
Let $S\in\mathcal F(C_n)$. If
$\mathsf h(S)\ge\lfloor |S|/2\rfloor$, then
$\mathsf N_0^n(S)\ge B(|S|,n)$.
\end{lemma}
\begin{proof}
The empty sequence is immediate, so assume $S\ne1$.
Translate a term of maximum multiplicity to zero. Translation does
not change the sum of any $n$ terms in $C_n$. Write $S=0^aT$ with
$0\notin\supp(T)$ and $|T|=c$. Thus $a+c=|S|$ and $a\ge c-1$.
Every selection of $n$ positions of $T$ contains a nonempty zero-sum
subsequence: among the $n+1$ partial sums of any ordering, two are
equal. Counting incidences between such selections and their
nonempty zero-sum subsequences gives
\[
 \binom cn\le\sum_{j=1}^{\min(n,c)}
            \mathsf N_0^j(T)\binom{c-j}{n-j}.
\]
As $a\ge c-j$ for $j\ge1$, it follows that
\[
 \mathsf N_0^n(S)
 =\binom an+\sum_{j=1}^{\min(n,c)}
                  \mathsf N_0^j(T)\binom a{n-j}
 \ge\binom an+\binom cn\ge B(a+c,n).
\]
The last inequality follows by balancing $a,c$, since
$\binom{x+1}{n}-\binom xn=\binom x{n-1}$ is nondecreasing in
$x\in\N_0$. If $c<n$, the incidence inequality has zero on the
left and remains valid.
\end{proof}

The following lemma is well known.
\begin{lemma}
\label{lem:nMZS}
Let $n\ge3$ and $U\in\mathcal F(C_n)$.
\begin{enumerate}
  \item If $U$ is minimal zero-sum and $|U|=n$, then $U=h^n$ for some $h\in C_n$ with order $n$.
  \item If $U$ is zero-sum free and $|U|=n-1$, then $U=h^{n-1}$ for some $h\in C_n$ with order $n$.
\end{enumerate}
\end{lemma}

\begin{lemma}
\label{lem:representations}
Let $n\ge3$ and $U\in\mathcal F(C_n^\bullet)$.
\begin{enumerate}
\item If $|U|=n$, then $\mathsf N_g(U)\ge2$ for every
$g\in\Sigma^*(U)$. If $U$ contains an element of order less than
$n$, then $\mathsf N_g(U)\ge3$ for every such $g$.
\item If $|U|=n-1$ and $U$ contains an element of order less than
$n$, then $\mathsf N_g(U)\ge2$ for every $g\in\Sigma^*(U)$.
\end{enumerate}
\end{lemma}
\begin{proof} First, consider the case $g=0$.

If $|U|=n$, then $U$ contains a nonempty zero-sum subsequence. This, together with the empty sequence, implies that $\mathsf N_0(U)\ge2$.

If $|U|=n-1$ and $U$ contains an element of order less than $n$, then $U$ is not zero-sum free by Lemma \ref{lem:nMZS}. So $U$ contains a nonempty zero-sum subsequence. Hence, $\mathsf N_0(U)\ge2$.

Suppose $|U|=n-1$ and $U$ contains an element say $a$, of order less than $n$. Let $V$ be a subsequence of $U$ containing $a$ and $|V|=n-1$. Then $V$ has a nonempty zero-sum subsequence, say $W$. Take a term $b$ of $W$ such that $U\cdot b^{-1}$ contains $a$. Again, $U\cdot b^{-1}$ contains a new nonempty zero-sum subsequence. These two zero-sum subsequence, together with the empty sequence, implies that $\mathsf N_g(U)\ge3$.

Next, let $g\in \Sigma^*(U)$ with $g\neq 0$. Let $W$ be a subsequence of $U$ such that $\sigma(W)=g$. Consider the new sequence: $U'=(-W) \cdot T$, where $T=UW^{-1}$. Then $|U'|=|U|$, and $U$ contains an element of order less than $n$ if and only if so does $U'$. Let $V$ be zero-sum subsequences of $U'$, and write $V=(-V_1)\cdot V_2$, where $-V_1$ and $V_2$ are (possible empty) subsequences of $-W$ and $T$, respectively. Then $V_1$ is a subsequence of $W$ and $\sigma(V_1)=\sigma(V_2)$. We obtain a new subsequence $W\cdot V_1^{-1}\cdot V_2$ of $U$, whose sum is $g$. Moreover, different zero-sum subsequences of $U'$ will lead to different $g$ sum subsequences of $U$. Hence, the lemma holds for $g\neq 0$.
\end{proof}

\section{Binomial estimates and the short range}\label{red:estimates}

Throughout this section, binomial coefficients with an integral lower
argument are extended to real upper arguments by
\[
 \binom{x}{r}=\frac{x(x-1)\cdots(x-r+1)}{r!}
 \qquad (r\in\N,\ x\ge r-1).
\]
We put
\[
 F(x)=x\log x-(x-1)\log(x-1)\quad(x>1),\qquad
 E(x)=F(x/2)-F(x)\quad(x>2).
\]
In particular,
\begin{equation}\label{red:derivatives}
 F'(x)=\log\frac{x}{x-1},\qquad
 F''(x)=-\frac1{x(x-1)},\qquad
 E'(x)=\frac12\log\left(1+\frac1{x(x-2)}\right).
\end{equation}
Thus $F$ is increasing and concave, and $E'$ is positive and decreasing.

\begin{lemma}\label{red:binomial-estimates}
Let $r\in\N$.
\begin{enumerate}
\item The function $x\mapsto\binom{x}{r}$ is convex on $[r-1,\infty)$.
The function
\[
 R_r(x)=\binom{x}{r}\exp\bigl(-rF(x/r)\bigr)
 \qquad(x>r)
\]
is decreasing, and its limit at infinity is $r^re^{-r}/r!$.
Consequently, if $L>2r$ is an integer, then
\begin{equation}\label{red:normalized-ratio}
 \frac{B(L,r)}{\binom Lr}\ge 2\exp\bigl(rE(L/r)\bigr).
\end{equation}
\item Let $n,\ell,a$ be positive integers, $N=\ell/(2n)>1$,
$A=a/n$, and $Q=\lceil\ell/2\rceil$. If $a\ge Q>n$, then
\begin{equation}\label{red:original-ratio}
 \log\frac{\binom an}{\binom Qn}
 \ge n\bigl(F(A)-F(N)\bigr)-F'(N)
       -\frac{1+1/N}{12n}.
\end{equation}
\end{enumerate}
\end{lemma}

\begin{proof}
The second derivative of the product defining $\binom{x}{r}$ is a sum
of products of nonnegative factors when $x\ge r-1$. This proves
convexity, including $r=1$. Moreover,
\[
 \frac{d}{dx}\log R_r(x)
 =\sum_{j=0}^{r-1}\frac1{x-j}-\log\frac{x}{x-r}\le0,
\]
because the sum is bounded above by $\int_{x-r}^{x}dt/t$.
The asserted limit follows either directly from the product or from
$F(u)=\log u+1+o(1)$ as $u\to\infty$. Convexity gives
$B(L,r)\ge2\binom{L/2}{r}$, and $R_r(L/2)\ge R_r(L)$ then gives
\eqref{red:normalized-ratio}.

For the second assertion we use the Stirling inequalities of
Robbins~\cite{Robbins1955},
\begin{equation}\label{red:stirling}
 \sqrt{2\pi k}\,(k/e)^k e^{1/(12k+1)}
 < k! < \sqrt{2\pi k}\,(k/e)^k e^{1/(12k)}
 \qquad(k\in\N).
\end{equation}
The first assertion and these inequalities imply
\[
 R_n(a)\ge\frac{e^{-1/(12n)}}{\sqrt{2\pi n}},
 \qquad
 R_n(Q)\le\frac{e^{1/(12Q)}}{\sqrt{2\pi n}}
                 \sqrt{\frac Q{Q-n}}.
\]
Since $Q/n\ge N$, it follows that
\[
 \log\frac{R_n(a)}{R_n(Q)}
 \ge-\frac12F'(N)-\frac{1+1/N}{12n}.
\]
Finally, $0\le Q-\ell/2\le1/2$ and the concavity of $F$ give
$nF(Q/n)\le nF(N)+F'(N)/2$. Combining these inequalities proves
\eqref{red:original-ratio}.
\end{proof}

We record the numerical estimates used below with explicit rational
bounds. This also fixes all constants at the endpoints of the ranges.

\begin{lemma}\label{red:log-certificate}
The following strict bounds hold:
\begin{align}
 \frac23F(87/10)-F(17/5)&>\frac9{1000},\label{red:short-constant}\\
 2\log2+\frac12\log(19/13)&<\frac{79}{50},\label{red:rounding-constant}\\
 F(4947/2500)+4F(4947/1250)-5F(17/5)&>\frac{21}{1000},
                                      \label{red:large-constant}\\
 \frac{21}{40}-E'(84/25)-\log(17/12)-\frac{11}{12750}
       &>\frac7{100}.                 \label{red:large-endpoint}
\end{align}
\end{lemma}

\begin{proof}
For $1\le x\le2$, put $z=(x-1)/(x+1)$ and
\[
 L(x)=2\sum_{j=0}^{15}\frac{z^{2j+1}}{2j+1},\qquad
 U(x)=L(x)+\frac{2z^{33}}{33(1-z^2)}.
\]
The series for $\log((1+z)/(1-z))$ gives
$L(x)\le\log x\le U(x)$, with strict inequalities when $x>1$.
For any positive rational $x$, write $x=2^ku$, where $k\in\Z$ and
$1\le u\le2$, and use $\log x=k\log2+\log u$. Substituting these
rational bounds, with the direction reversed for negative
coefficients, gives the following table. Each entry denotes the
expression on the corresponding left-hand side above.
\[
\begin{array}{c|r|r}
 \text{expression}&\text{strict lower bound}&\text{strict upper bound}\\ \hline
 \eqref{red:short-constant}&9297/10^6&9298/10^6\\
 \eqref{red:rounding-constant}&1576039/10^6&1576040/10^6\\
 \eqref{red:large-constant}&21197/10^6&21198/10^6\\
 \eqref{red:large-endpoint}&76881/10^6&76882/10^6
\end{array}
\]
All entries involve only the displayed finite sums of rational numbers
and their explicit remainder bounds. They imply the assertions.
\end{proof}

\begin{lemma}\label{lem:short-range}
Let $n\ge180$, or let $n\in\{105,165\}$. If $S\in\mathcal F(C_n)$
and $|S|\le\lfloor34n/5\rfloor$, then
$\mathsf N_0^n(S)\ge B(|S|,n)$.
\end{lemma}

\begin{proof}
Put $\ell=|S|$. If $\ell\le2n-2$, then $B(\ell,n)=0$ and the
assertion is immediate. Theorem~\ref{thm:grynkiewicz} gives the
assertion for $2n-1\le\ell\le19n/3$. For the remaining lengths it suffices, by the
same theorem, to prove
\begin{equation}\label{red:published-comparison}
 \binom{\ell-n}{r}\ge B(\ell,n),\qquad
 r=\left\lceil\frac{2n-1}{3}\right\rceil.
\end{equation}
First suppose $n\ge180$. Set $D=\ell-n$, $h=\lceil\ell/2\rceil$,
$t=\ell/n$, and
\[
 H(x,y)=x\log x-y\log y-(x-y)\log(x-y).
\]
Throughout $19/3<t\le34/5$. Since
$2n/3-1/3\le r\le2n/3+1/3$ and
$\partial H(D,y)/\partial y=\log((D-y)/y)$, we have
\[
 H(D,r)\ge H(D,2n/3)-\frac13\log8.
\]
Indeed the derivative is positive and less than $\log8$ throughout
the possible interval between $r$ and $2n/3$. Likewise,
\[
 H(h,n)\le nF(t/2)+\frac12\log(19/13).
\]
The function
\[
 G(t)=\frac23F\bigl(3(t-1)/2\bigr)-F(t/2)
\]
is decreasing for $t>2$. To see this, differentiate and square the
positive quantities in
\[
 G'(t)=\log\frac{t-1}{t-5/3}-\frac12\log\frac t{t-2}<0;
\]
the difference of the cross-multiplied expressions is
$2(3t^2-10t+9)/9>0$. Thus
\[
 H(D,r)-H(h,n)
 \ge nG(34/5)-\frac13\log8-\frac12\log(19/13).
\]
By \eqref{red:stirling},
\begin{align*}
 \frac{\binom Dr}{2\binom hn}
 &\ge\frac12\sqrt{\frac{Dn(h-n)}{hr(D-r)}}\,
 \exp\left(H(D,r)-H(h,n)
 -\frac1{12r}-\frac1{12(D-r)}-\frac1{12h}\right).
\end{align*}
The square-root factor is at least $1$: indeed $D/(D-r)>1$,
$h\ge19n/6$, and $r/n\le13/19$. Moreover,
$r\ge n/2$, $D-r\ge4n$, and $h\ge3n$, so that the sum of the
three error terms is less than $1/(4n)$. Therefore
\[
 \log\frac{\binom Dr}{2\binom hn}
 \ge nG(34/5)-2\log2-\frac12\log(19/13)-\frac1{4n}
 >\frac{9n}{1000}-\frac{79}{50}-\frac1{4n}>0.
\]
The last inequality holds at $n=180$ and its right-hand side increases
with $n$. Since $B(\ell,n)\le2\binom hn$, this proves
\eqref{red:published-comparison}.

For $n\in\{105,165\}$ we give four exact endpoint comparisons.
Here $r=2n/3$. Define
\[
 T_n(\ell)=\frac{\binom{\ell-n}{2n/3}}
                      {2\binom{\lceil\ell/2\rceil}{n}}.
\]
For a fixed parity write $\ell=2h-e$, where $e\in\{0,1\}$, and set
$v=2h-e-5n/3+1$. Direct cancellation gives
\[
 \frac{T_n(\ell+2)}{T_n(\ell)}
 =\frac{(v+r)(v+r+1)(h+1-n)}{v(v+1)(h+1)}<1.
\]
For completeness, the denominator minus the numerator, multiplied by
$9/n$, is
\[
 P_e(h)=12h^2+(12-24e-20n)h+9n^2+(18e-11)n+9e^2-15e.
\]
It is increasing for $h\ge19n/6$, and at $19n/6$ its value is
$66n^2+(27-58e)n+9e^2-15e>0$. The denominators are positive in the
range in question. It remains to check the largest length of each
parity. In the following table, each row satisfies
$k<1000T_n(\ell)<k+1$:
\begin{equation}\label{red:finite-table}
\begin{array}{c|r|r|r|r}
 n&\ell&D=\ell-n&h=\lceil\ell/2\rceil&k\\ \hline
 105&714&609&357&1451\\
 105&713&608&357&1284\\
 165&1122&957&561&2535\\
 165&1121&956&561&2244
\end{array}
\end{equation}
These are integer-product comparisons: an entry with lower endpoint
$k$ means exactly
\[
 2k\,r!\prod_{j=0}^{n-1}(h-j)
 <1000\,n!\prod_{j=0}^{r-1}(D-j)
 <2(k+1)\,r!\prod_{j=0}^{n-1}(h-j).
\]
Thus the table is a finite rational certificate, with all its factors
specified. Every lower endpoint exceeds $1000$, which proves
\eqref{red:published-comparison} for these two values of $n$.
\end{proof}

\section{Doubling the group order}\label{red:doubling-section}

The doubling step rests on an extremal property of two-color
matchings. A maximum matching on an odd set leaves one vertex unmatched.
All matchings below are chosen uniformly from the maximum matchings.

\begin{lemma}\label{red:matching}
Color a set of $\ell$ vertices with two colors, put
$t=\lfloor\ell/2\rfloor$, and let $M$ be the number of monochromatic
pairs in a uniformly chosen maximum matching. For $1\le b\le t$,
\begin{equation}\label{red:matching-bound}
 \mathbb E B(M,b)\ge B(\ell,2b)
                       \frac{\binom tb}{\binom\ell{2b}}.
\end{equation}
\end{lemma}

\begin{proof}
We first show that the expectation is minimized when the two color
classes have sizes $\lfloor\ell/2\rfloor$ and
$\lceil\ell/2\rceil$.

Suppose $\ell=2t$, and let the class sizes be $a\le c=2t-a$.
Write $D(2j)=(2j-1)!!$, with $D(0)=1$. If $T$ is the number of
bichromatic pairs, then $M=t-T$, and direct enumeration gives
\[
 p_a(u):=\mathbb P(T=u)
 =\frac{a!c!}{D(2t)\,u!\,2^{t-u}
                    ((a-u)/2)!((c-u)/2)!}
\]
when $0\le u\le a$ and $u\equiv a\pmod2$; the probability is
zero otherwise. For $a+2\le c$,
\[
 \frac{p_{a+2}(u)}{p_a(u)}
 =\frac{(a+2)(a+1)}{c(c-1)}\frac{c-u}{a+2-u}.
\]
This ratio increases with $u$ on the common support, and any new
support point occurs at its upper end. Hence the difference of the
two probability mass functions changes sign at most once, from
negative to positive. Since their total masses agree, their upper
tails are ordered: $T$ increases stochastically when $a$ is increased
by two towards $t$. Thus the minimum of $\mathbb E B(M,b)$ occurs at
one of the splits $(t,t)$ and $(t-1,t+1)$.

For an integer $j\ge0$ set
\[
 A_j(a)=\binom a{2j}\frac{D(2j)D(2t-2j)}{D(2t)}.
\]
This is $\mathbb E\binom Xj$, where $X$ is the number of pairs
contained in a specified class of size $a$. At the split $(t,t)$,
the numbers of monochromatic pairs in the two classes agree, so the
expectation is $2A_b(t)$. At the split $(t-1,t+1)$ these numbers are
$X$ and $X+1$, so the expectation is
$2A_b(t-1)+A_{b-1}(t-1)$. Factorial cancellation yields
\[
 2A_b(t-1)+A_{b-1}(t-1)-2A_b(t)
 =\frac{2b-1}{2t-2b+1}A_{b-1}(t-1)\ge0.
\]
The equality remains valid when one of the binomial coefficients
vanishes. This proves the assertion for even $\ell$.

Now suppose $\ell=2t+1$ and the class sizes are $a\le c-1$.
Recolor one fixed vertex $x$ of the larger class. In a uniform
maximum matching, $x$ is unmatched with probability $1/\ell$ and
is matched to each other specified vertex with probability $1/\ell$.
If it is matched to its own color, the number of monochromatic pairs
changes from $R+1$ to $R$; the residual class sizes are $(a,c-2)$.
If it is matched to the other color, that number changes from $R$ to
$R+1$; the residual sizes are $(a-1,c-1)$. It follows that, for every
integer $s$, the old upper tail minus the new upper tail is
\begin{equation}\label{red:odd-tail}
 \frac{(c-1)\mathbb P_{a,c-2}(R=s-1)
              -a\mathbb P_{a-1,c-1}(R=s-1)}{\ell}.
\end{equation}
Here the subscripts indicate the class sizes of a uniform maximum
matching on the $2t-1$ residual vertices.

To compare these probabilities, for a split $(d,f)$ of an odd number
$v$ the probability of $u$ bichromatic pairs is
\[
 \frac{d!f!}{v!!\,u!\,2^{(v-1)/2-u}
       ((d-u-\epsilon)/2)!((f-u-1+\epsilon)/2)!},
 \qquad \epsilon\equiv d-u\pmod2,\quad \epsilon\in\{0,1\}.
\]
Set $u=t-s$ and $\delta\equiv a-u\pmod2$, with
$\delta\in\{0,1\}$. For probabilities on the common support, this
formula gives
\[
 \frac{\mathbb P_{a,c-2}(R=s-1)}
      {\mathbb P_{a-1,c-1}(R=s-1)}
 =\begin{cases}
    a/(c-1),&\delta=1,\\[2pt]
    \displaystyle\frac a{c-1}\frac{c-1-u}{a-u},&\delta=0.
   \end{cases}
\]
Since $c-1\ge a$, \eqref{red:odd-tail} is nonnegative. At the
additional support point $u=a$, the second probability is zero;
when $a=0$ the assertion is immediate. Thus recoloring towards a
balanced split makes $M$ stochastically smaller. As $B(M,b)$
increases with $M$, its expectation is minimized at $(t,t+1)$.

It remains to evaluate the expectation at a balanced split. Let
$X,Y$ count the pairs contained in its two classes. Then $X=Y$ if
$\ell$ is even, and $|X-Y|\le1$ if $\ell$ is odd. Hence
\[
 B(M,b)=\binom Xb+\binom Yb.
\]
Every fixed set of $2b$ vertices is the union of $b$ pairs with the
same probability
\[
 p=\frac{\binom tb}{\binom\ell{2b}}:
\]
each matching has exactly $\binom tb$ such unions, and all
$2b$-sets are symmetric. Counting those unions contained in either
balanced class gives $\mathbb E B(M,b)=pB(\ell,2b)$. This proves
\eqref{red:matching-bound}.
\end{proof}

\begin{lemma}\label{lem:doubling}
For every $b\in\N$, the assertion $\mathcal C(b)$ implies
$\mathcal C(2b)$.
\end{lemma}

\begin{proof}
Let $S\in\mathcal F(C_{2b})$, put $\ell=|S|$, and assume
$\ell\ge2b$, since otherwise the required bound is zero. Color
its positions according to their images in $C_2$, and choose a
uniform maximum matching. Retain its $M$ monochromatic pairs.
Each retained pair has sum in the subgroup $2C_{2b}\cong C_b$;
its image under this isomorphism will be called its label.
By $\mathcal C(b)$ the sequence of labels has at least $B(M,b)$
zero-sum subsequences of length $b$. Their unions are distinct
zero-sum sets of $2b$ positions of $S$.

A fixed set of $2b$ positions is the union of $b$ matching pairs
with probability $p=\binom{\lfloor\ell/2\rfloor}b/\binom\ell{2b}$.
The further requirement that all these pairs are monochromatic can
only decrease this probability. Thus, by double counting and
Lemma~\ref{red:matching},
\[
 p\mathsf N_0^{2b}(S)\ge\mathbb E B(M,b)\ge pB(\ell,2b).
\]
Division by $p>0$ proves the assertion.
\end{proof}

\section{Lifting from a large residue class}\label{red:large-section}

The next reduction applies whenever one class modulo a divisor of
the group order is sufficiently large. No primality assumption on
that divisor is required.

\begin{lemma}\label{lem:large-class}
Let $q\ge5$ and $b\ge25$ be integers, put $n=qb$, and suppose
$\mathcal C(b)$ holds. Let $S\in\mathcal F(C_n)$ have length
$\ell\ge34n/5$. If some class modulo $q$ contains at least
\[
 \left(\frac12+\frac{41}{100q}\right)\ell
\]
positions of $S$, then $\mathsf N_0^n(S)\ge B(\ell,n)$.
\end{lemma}

\begin{proof}
Let $a$ be the size of the specified class and put $L=\lfloor a/q\rfloor$.
Choose $qL$ of its positions uniformly and partition them uniformly
into $L$ blocks of size $q$. Every block sum belongs to
$qC_n\cong C_b$. Applying $\mathcal C(b)$ to the resulting labels
gives at least $B(L,b)$ zero-sum unions of $b$ blocks. By symmetry,
every fixed $n$-set in the class is the union of $b$ blocks with
probability $\binom Lb/\binom an$. Consequently,
\begin{equation}\label{red:class-count}
 \mathsf N_0^n(S)\ge
 Z_*:=\frac{B(L,b)\binom an}{\binom Lb}.
\end{equation}

Set
\[
 N=\frac\ell{2n}\ge\frac{17}5,\qquad
 A=\frac an,\qquad \lambda=\frac Lb,\qquad
 \alpha=\frac{41}{50}.
\]
Then $A\ge(1+\alpha/q)N$ and $A-1/b<\lambda\le A$.
In particular $a\ge\lceil\ell/2\rceil$ and $L>2b$.
Lemma~\ref{red:binomial-estimates}, together with
$B(\ell,n)\le2\binom{\lceil\ell/2\rceil}n$, gives
\begin{equation}\label{red:class-entropy}
 \log\frac{Z_*}{B(\ell,n)}
 \ge bJ_q(A,N)-E'(\lambda)-F'(N)
                 -\frac{1+1/N}{12qb},
\end{equation}
where
\[
 J_q(A,N)=F(A/2)+(q-1)F(A)-qF(N).
\]
Here we used
$b(E(A)-E(\lambda))\le E'(\lambda)$, which follows from the
monotonicity of $E'$ and $A-\lambda<1/b$.

We claim that, throughout the present range,
\begin{equation}\label{red:uniform-margin}
 J_q(A,N)\ge F(4947/2500)+4F(4947/1250)-5F(17/5)
                      >\frac{21}{1000}.
\end{equation}
The function $J_q(A,N)$ increases with $A$. It is therefore enough
to consider
\[
 H_q(N)=F(rN/2)+(q-1)F(rN)-qF(N),
 \qquad r=1+\frac\alpha q.
\]
The convergent expansion
\[
 F(x)=\log x+1-\sum_{j\ge1}\frac1{j(j+1)x^j}
 \qquad(x>1)
\]
gives
\[
 H_q(N)=q\log r-\log2-
 \sum_{j\ge1}
 \frac{(2/r)^j+(q-1)(1/r)^j-q}{j(j+1)N^j}.
\]
Each numerator is positive. Indeed, the weighted mean of $2/r$
and $1/r$, with weights $1/q$ and $(q-1)/q$, equals
$(q+1)/(q+\alpha)>1$, and convexity of $x\mapsto x^j$ applies.
Thus $H_q(N)$ increases with $N$.

We next regard $q\ge5$ as real. Write $t=1/q$ and
$x=N+\alpha Nt$. The fundamental theorem of calculus gives
\[
 H_q(N)=E(x)+\alpha N\int_0^1 F'(N+\alpha Ntu)\,du.
\]
Its derivative with respect to $t$, divided by $\alpha N$, is
\[
 E'(x)+\alpha N\int_0^1uF''(N+\alpha Ntu)\,du
 \le\frac1{2x(x-2)}-\frac{\alpha N}{2x(x-1)}\le0.
\]
For the last inequality it suffices that
$x-1\le\alpha N(x-2)$. Since $x\ge N$ and $\alpha N>1$, this
follows from $N-1\le\alpha N(N-2)$; the latter holds at $N=17/5$
and its difference increases thereafter. Hence $H_q(N)$ increases
with $q$. Its minimum is therefore $H_5(17/5)$, which is the
expression in \eqref{red:uniform-margin}. Lemma~\ref{red:log-certificate}
proves its strict lower bound.

Finally, $\lambda>17/5-1/b$, and \eqref{red:class-entropy} yields
\[
 \log\frac{Z_*}{B(\ell,n)}
 >\frac{21b}{1000}-E'(17/5-1/b)-F'(17/5)
                       -\frac{1+5/17}{60b}.
\]
Every term on the right is increasing as a function of $b\ge25$.
At $b=25$ the right-hand side is the expression in
\eqref{red:large-endpoint}, and is greater than $7/100$.
It follows that $Z_*>B(\ell,n)$, and
\eqref{red:class-count} completes the proof.
\end{proof}

\section{The tripling reduction}\label{sec:tripling}

The passage from $C_b$ to $C_{3b}$ is the main step of the proof.
There are two kinds of triples whose sums vanish in $C_3$: triples
contained in one residue class and triples meeting all three classes.
The latter will be called transversal triples.  We combine the two
kinds, but vary their numbers before choosing the triples themselves.
The variation balances the probabilities with which different
subsequences are counted.  This is the reason for introducing inverse
binomial coefficients below.

Throughout this section, $b\ge35$, $n=3b$, and $\mathcal C(b)$ is
assumed.  Let $S\in\Fcal(C_{3b})$, let $\ell=|S|$, and let
$a\ge d\ge c\ge0$ be the sizes of its residue classes modulo $3$.
We work in the range
\begin{equation}\label{tri:range}
 \ell=a+d+c\ge\frac{102}{5}b.
\end{equation}
Permuting these three class sizes causes no difficulty: every
permutation of $C_3$ is affine, and both kinds of triples still have
sum zero in $C_3$.

\subsection{Binomial estimates}

For an integer $k\ge1$ and a real number $y>k-1$, put
\[
 f_k(y)=\binom yk=\frac{y(y-1)\cdots(y-k+1)}{k!},\qquad
 H_k(y)=\sum_{i=0}^{k-1}\frac1{y-i},\qquad
 \rho_k(y)=\frac{\sum_{i=0}^{k-1}(y-i)^{-2}}{H_k(y)^2}.
\]
Thus $(\log f_k)'=H_k$ and $H_k'=-\rho_kH_k^2$.
The function $f_k$ is increasing and convex on $[k-1,\infty)$,
and $\log f_k$ is concave on $(k-1,\infty)$.
We also use
\[
 F(x)=x\log x-(x-1)\log(x-1)\quad(x>1),\qquad F(1)=0.
\]
All logarithms in this section are natural.

\begin{lemma}\label{tri:analytic}
For $y>k$ one has
\begin{align}
 \log\frac{y+1}{y+1-k}&\le H_k(y)\le\log\frac y{y-k},
 &\frac1k&\le\rho_k(y)\le\frac{y}{k(y-k+1)}.\label{tri:Hr}
\end{align}
The function $\rho_k$ is decreasing.  Moreover,
\begin{align}
 f_k(y)&\ge\frac{\exp(kF(y/k)-1/(12k))}{\sqrt{2\pi k}},
 \label{tri:real-lower}\\
 f_k(Q)&\le\frac{\exp(kF(Q/k)+1/(12Q))}{\sqrt{2\pi k}}
             \sqrt{\frac Q{Q-k}}\quad(Q\in\N,\ Q>k).
 \label{tri:integer-upper}
\end{align}
For integers $M>k$,
\begin{equation}\label{tri:integer-lower}
 \log\binom Mk\ge kF(M/k)-\tfrac12\log(2\pi k)
       +\tfrac12F'(M/k)-\frac{1+1/(M/k-1)}{12k}.
\end{equation}
If $A\ge D>k$, then
\begin{equation}\label{tri:entropy-ratio}
 \frac{f_k(A)}{f_k(D)}\le
       \exp\bigl(k[F(A/k)-F(D/k)]\bigr).
\end{equation}
Finally, for an integer $L$ with $L/2>k-1$,
\begin{equation}\label{tri:parity}
 2f_k(L/2)\le B(L,k)
 \le2f_k(L/2)\exp\bigl(H_k(L/2)^2/8\bigr).
\end{equation}
\end{lemma}
\begin{proof}
The two bounds for $H_k$ are integral comparisons.  Cauchy's
inequality gives $\rho_k\ge1/k$.  If $u_i=(y-i)^{-1}$, then
$\sum u_i^2\le(\max u_i)\sum u_i$ and
$\sum u_i\ge k/y$, proving the upper bound.  Writing
$S_j=\sum_i(y-i)^{-j}$, differentiation shows that $\rho_k'$ has
the sign of $S_2^2-S_1S_3$, which is nonpositive by Cauchy's
inequality.

The function
\[
 R_k(y)=f_k(y)\exp(-kF(y/k))
\]
is decreasing, since $(\log R_k)'=H_k(y)-\log(y/(y-k))\le0$.
Its limit at infinity is $k^ke^{-k}/k!$.  The factorial bounds
of Robbins \cite{Robbins1955}
\[
 \sqrt{2\pi r}(r/e)^r\le r!
 \le\sqrt{2\pi r}(r/e)^r e^{1/(12r)}\qquad(r\in\N)
\]
give \eqref{tri:real-lower}--\eqref{tri:integer-lower}; monotonicity
of $R_k$ gives \eqref{tri:entropy-ratio}.  Convexity gives the
first inequality in \eqref{tri:parity}.  For odd $L$, the tangent
inequalities for $\log f_k$ at $L/2$ give
$B(L,k)\le2f_k(L/2)\cosh(H_k(L/2)/2)$.  Now use
$\log\cosh t\le t^2/2$.  The even case is immediate.
\end{proof}

We record explicitly how the few numerical constants used below
are obtained.  This also permits all the estimates to be read as
exact inequalities, rather than numerical approximations.  For
$1\le x\le2$, set
\begin{equation}\label{tri:log-certificate}
 L_{120}(x)=2\sum_{j=0}^{119}\frac{z^{2j+1}}{2j+1},\qquad
 z=\frac{x-1}{x+1}.
\end{equation}
The geometric series, integrated term by term, gives
\[
 L_{120}(x)\le\log x\le L_{120}(x)
       +\frac{2z^{241}}{241(1-z^2)}.
\]
For an arbitrary positive rational $x$, first write $x=2^ru$ with
$1\le u<2$ and use $\log x=r\log2+\log u$.  Thus these bounds
involve rational arithmetic only.  In subsequent displays, the
following strict inequalities are used.  Each follows by substituting
\eqref{tri:log-certificate}; the last column is a rational bound on
the indicated expression.
\begin{equation}\label{tri:constant-table}
\begin{array}{c|c}
 \text{expression}&\text{bound}\\ \hline
 2F(13/3)-3F(2)&>0\\
 3F(12/5)-2F(22/5)&>0\\
 4F(68/25)-3F(17/5)&>0\\
 F(11/5)-\tfrac12F(22/5)
       -\dfrac{F(22/5)^2}{48(\log2)^2}&>191/2000\\
 F(51/50)+F(51/5)+F(204/25)-3F(17/5)&>1/5\\
 F(51/40)+F(51/5)-2F(51/20)&>1/2\\
 F(272/125)+2F(544/125)-3F(17/5)&>27/2000\\
 \log(19/13)&<19/50\\
 \log6&<9/5
\end{array}
\end{equation}

\begin{lemma}\label{tri:rounding}
Let $x\ge k\ge b$, where $k\in\N$ and $x\in[k,k+1]$.
There is a random variable $U\in\{k,k+1\}$ such that
$\mathbb E f_b(U)=f_b(x)$.  It satisfies
\begin{equation}\label{tri:round-loss}
 0\le x-\mathbb E U\le\frac{b-1}{8(k-b+1)}.
\end{equation}
More generally, the right side may be replaced by $M/8$ whenever
$(\log f_b')'\le M$ on $[k,k+1]$.  If $w\in\N$ and $b\le w\le x$,
then
\begin{equation}\label{tri:interpolation}
 \mathbb E B(U+w,b)\ge2\binom{(x+w)/2}{b}.
\end{equation}
\end{lemma}
\begin{proof}
Take $\Pr(U=k+1)=(f_b(x)-f_b(k))/(f_b(k+1)-f_b(k))$ and put
$t=x-k$.  Each product occurring in $f_b'$ has logarithmic
derivative at most $(b-1)/(k-b+1)$.  If $(\log f_b')'\le M$,
the quotient $f_b'(k+s)e^{-Ms}$ is decreasing.  Comparing its
integrals on $[0,t]$ and $[t,1]$ gives
\[
 \Pr(U=k+1)\ge\frac{e^{Mt}-1}{e^M-1}.
\]
Convexity gives $\Pr(U=k+1)\le t$.  The maximum of
$t-(e^{Mt}-1)/(e^M-1)$, with $h=M/2$, is
\[
 \frac{h\coth h-1+\log(\sinh h/h)}{2h}\le\frac h4=\frac M8.
\]
Indeed $\coth h-1/h\le h/3$, as follows by comparing the power
series of $(1+h^2/3)\sinh h$ and $h\cosh h$; integration gives
$\log(\sinh h/h)\le h^2/6$.

For \eqref{tri:interpolation}, write $u=\lfloor x\rfloor$,
$t=x-u$, $j=\lfloor(u+w)/2\rfloor$, and
\[
 A_v(s)=\frac{f_b(v+s)-f_b(v)}{f_b(v+1)-f_b(v)}
       \quad(0\le s\le1).
\]
All roots of $f_b'$ are real and at most $b-1$, by Rolle's
theorem.  Thus $f_b'$ is positive and log-concave above $b-1$.
Since $j\le u$, the ratio $f_b'(j+s)/f_b'(u+s)$ is increasing,
and comparison of normalized integrals gives $A_j(s)\le A_u(s)$.
Also $A_u$ is convex.  When $u+w$ is even,
\eqref{tri:interpolation} is equivalent to
$2A_j(t/2)\le A_u(t)$.  When $u+w$ is odd, it is equivalent to
$2[1-A_j((1+t)/2)]\ge1-A_u(t)$.  Both follow from the preceding
comparison and convexity.
\end{proof}

\begin{lemma}\label{tri:midpoint}
Suppose $d>6b$, $a\ge d$, $x>b+1$, and
$\binom xb=\binom a{3b}/\binom d{3b}$.  Then
\begin{equation}\label{tri:midpoint-gap}
 \log\binom{(x+b)/2}{b}
 -\log\frac{\binom{(a+d)/2}{3b}}{\binom d{3b}}
 \ge\Delta_b:=\frac1{48}-\frac1{768b}.
\end{equation}
\end{lemma}
\begin{proof}
Put $W=\log\binom xb>\log(b+1)$ and
$L_\alpha(W)=\alpha^{-1}\log((1+e^{\alpha W})/2)$.
The logarithmic derivative of the derivative of the inverse of
$\log f_k$ is $\rho_k$.  A positive density $p$ on $[0,W]$
with $(\log p)'\ge\alpha$ has its median at least
$L_\alpha(W)$: divide by $e^{\alpha s}$ and compare the
integrals before and after the median.  The reverse derivative
bound reverses the conclusion.

The two logarithms in \eqref{tri:midpoint-gap} are the median
coordinates of the derivatives of the corresponding inverse
functions.  By \eqref{tri:Hr}, $\rho_{3b}(y)<\beta:=2/(3b)$
for $y\ge d$, and $\rho_b\ge1/b$.
If $W^2\ge b/2$, monotonicity in $W$ of
$L_{1/b}(W)-L_\beta(W)$ and
\[
 t^2/2-t^4/12\le\log\cosh t\le t^2/2
\]
give, on setting $z=1/\sqrt{2b}$,
\[
 L_{1/b}(W)-L_\beta(W)
 \ge b\{\log\cosh(z/2)-\tfrac32\log\cosh(z/3)\}
 \ge\Delta_b.
\]
If $W^2\le b/2$, put $h=x-b>1$.  Since
\[
 W\ge h\sum_{j=1}^b\frac1{h+j},\qquad
 W^2\rho_b(x)\ge\sum_{j=1}^b\frac{h^2}{(h+j)^2}
 \ge\sum_{j=1}^b\frac1{(j+1)^2}>\frac35,
\]
we may use $\alpha=3/(5W^2)$ on the entire interval, because
$\rho_b$ decreases.  The last strict inequality follows by
retaining $1/4+1/9+1/16+1/25$ and bounding the remaining terms
below by $\int_6^{b+2}t^{-2}\,dt$.  The same bounds for
$\log\cosh$ now give
\[
 L_\alpha(W)-L_\beta(W)
 \ge\frac{3/5-\beta W^2}{8}-\frac{(3/5)^3}{192W^2}
 >\frac1{30}-\frac1{8000}>\frac1{48}.
\]
Here $\log(b+1)\ge\log36>3$.  This proves the lemma.
\end{proof}

\subsection{Counting unions of triples}

For nonnegative integers $r,u,v$, choose uniformly a family of
disjoint triples consisting of $r$ transversal triples, $u$ triples
from the class of size $a$, and $v$ triples from the class of size
$d$.  Such a family exists precisely when
\begin{equation}\label{tri:capacity}
 r\le c,\qquad r+3u\le a,\qquad r+3v\le d.
\end{equation}
The sum of each triple belongs to $3C_{3b}\cong C_b$.
Applying $\mathcal C(b)$ to these sums supplies at least
$B(r+u+v,b)$ zero-sum unions of $b$ triples.  Each such union is
a subsequence of $S$ of length $3b$, with positions distinguished.

\begin{lemma}\label{tri:hybrid}
Assume $a,d\ge3b$, $c>0$, and that each of $u,v$ is either zero
or at least $b$.  Put
\[
 \kappa^b=\frac{(3b)!}{(b!)^2(6b)^b},\quad
 R=\frac1{\binom ab\binom db},\quad
 A=\frac{\binom ub}{\binom a{3b}},\quad
 D=\frac{\binom vb}{\binom d{3b}}.
\]
For $j+h+s=b$, the probability that a fixed set of positions with
class sizes $(3j+s,3h+s,s)$ is a union of retained triples is at most
\begin{equation}\label{tri:hybrid-prob}
 A^{j/b}D^{h/b}\bigl(\kappa^bR(r/c)^b\bigr)^{s/b}.
\end{equation}
Consequently, a finite mixture of feasible triples $(r,u,v)$ with
\begin{equation}\label{tri:balanced-mean}
 \mathbb E A=\mathbb E D=\alpha>0,\qquad
 \kappa^bR\,\mathbb E(r/c)^b\le\alpha
\end{equation}
satisfies
\begin{equation}\label{tri:master-count}
 \mathsf N_0^{3b}(S)\ge
       \frac{\mathbb E B(r+u+v,b)}{\alpha}.
\end{equation}
For a single type, the same conclusion holds with denominator
$\max\{A,D,\kappa^bR(r/c)^b\}$.  If $r=0$, the third endpoint
is omitted; in this form the assertion also applies when $c=0$.
\end{lemma}
\begin{proof}
By symmetry, the exact probability in question is
\[
 P(j,h,s)=\frac{\binom rs\binom uj\binom vh}
 {\binom cs\binom a{3j+s}\binom d{3h+s}}.
\]
For $u\ge b$, comparison of the products $(u-i)/(b-i)$ gives
$\binom uj\le\binom bj\binom ub^{j/b}$.  Concavity of
$\log(a)_k$ as a function of $k$ gives
\[
 \binom a{3j+s}\ge
 \binom a{3b}^{j/b}\binom ab^{s/b}
 \frac{(3b)!^{j/b}(b!)^{s/b}}{(3j+s)!}.
\]
The analogous inequalities hold for $v,d,h$.  It remains to bound
the factorial factor
\[
 K(j,h,s)=\frac{\binom bj\binom bh(3j+s)!(3h+s)!}
 {(3b)!^{1-s/b}(b!)^{2s/b}}.
\]
Vandermonde's identity implies
\[
 \binom{3b-s}{3j+s}\ge
 \binom{2b}{2j+s}\binom{b-s}{j}
 \ge\binom bj\binom bh\frac{\binom{2b}s}{\binom bs}.
\]
For the second inequality, assume $j\le h$ and cancel factorials.
The remaining ratio is
\[
 \prod_{i=1}^j
 \frac{(2h+s+i)(2h+s+j+i)}{(j+s+i)(h+s+i)}\ge1.
\]
It follows that $K(j,h,s)\le K_s$, where
\[
 K_s=\frac{\binom bs}{\binom{2b}s}
       \frac{(3b-s)!}{(3b)!^{1-s/b}(b!)^{2s/b}}.
\]
Writing $H=((3b)!/(b!)^2)^{1/b}$, we have $K_0=1$ and
\[
 \frac{K_{s+1}}{K_s}
 =\frac{H(b-s)}{(2b-s)(3b-s)}\le\frac H{6b}=\kappa.
\]
The last rational function decreases for $0\le s<b$.
Finally $\binom rs/\binom cs\le(r/c)^s$, proving
\eqref{tri:hybrid-prob}.  Zero block counts are handled by
omitting their zero-exponent factors.

H\"older's inequality and \eqref{tri:balanced-mean} give
$\mathbb E P(j,h,s)\le\alpha$.  Summing eligibility
probabilities over all zero-sum sets of positions, and comparing
with the expected number supplied by $\mathcal C(b)$, proves
\eqref{tri:master-count}.
\end{proof}

Two simpler consequences will be used repeatedly.  A family of
$c$ transversal triples gives
\begin{equation}\label{tri:rainbow}
 \mathsf N_0^{3b}(S)\ge B(c,b)\binom ab\binom db.
\end{equation}
If $L=\lfloor a/3\rfloor$, triples from the largest class alone give
\begin{equation}\label{tri:one-class}
 \mathsf N_0^{3b}(S)\ge
       B(L,b)\frac{\binom a{3b}}{\binom Lb}.
\end{equation}
Both follow by the same symmetry and expectation argument, without
an application of H\"older's inequality.

The following estimate provides the required endpoint probabilities.
Put $\mu=(b-1)/2$.  For $u\ge b$ and $a\ge d\ge3b$,
\begin{equation}\label{tri:midpoint-product}
 \frac{\binom ub/\binom d{3b}}{\kappa^bR}
 \ge\left\{
 \frac{12b(u-\mu)(a-\mu)}
 {(b+1)(d-2b)(d-2b+1)}\right\}^{b}.
\end{equation}
Indeed the ratio on the left is
\[
 \prod_{i=0}^{b-1}
 \frac{6b(u-i)(a-i)}{(b-i)(d-b-2i)(d-b-2i-1)}.
\]
The logarithm of its factor, considered as a function of real $i$,
has positive second derivative:
\[
 \frac1{(b-i)^2}-\frac1{(u-i)^2}-\frac1{(a-i)^2}
 +\frac4{(d-b-2i)^2}+\frac4{(d-b-2i-1)^2}>0.
\]
Jensen's inequality at $i=\mu$ proves
\eqref{tri:midpoint-product}.

\subsection{Two classes and a large residue class}

We first establish the entropy comparison needed when a proposed
number of monochromatic triples exceeds its capacity.

\begin{lemma}\label{tri:entropy-root}
Let $N\ge17/5$, and let $A_0+D_0=2N$, $A_0\ge D_0\ge1$,
with $2F(A_0)=3F(D_0)$.  Then
\begin{equation}\label{tri:root-gap}
 F(A_0/2)+2F(A_0)-3F(N)>\delta,
 \qquad \delta=191/2000.
\end{equation}
\end{lemma}
\begin{proof}
The left side of $2F(A)-3F(2N-A)=0$ increases strictly in $A$,
so the root is unique.  It satisfies $D_0>2$: at $D=2$ the
expression is at least $2F(13/3)-3F(2)>0$.
It also satisfies $A_0>N+1$, since
$2F(N+1)-3F(N-1)$ decreases with $N$ and is negative at $17/5$,
by \eqref{tri:constant-table}.

Let $G=F^{-1}$.  Then
\[
 (\log G')'(F(x))=
 \frac1{x(x-1)\log^2(x/(x-1))}.
\]
This decreases for $x>1$: the derivative of its reciprocal has
the sign of $(2x-1)\log(x/(x-1))-2>0$, using
$\log t>2(t-1)/(t+1)$.  Hence it is at most
$k_0=1/(2\log^22)$ on the interval in question.  The density
comparison used in Lemma~\ref{tri:midpoint} yields
\[
 F(N)-F(D_0)\le
 \frac1{k_0}\log\frac{1+\exp(k_0F(A_0)/3)}2.
\]
Therefore the expression in \eqref{tri:root-gap} is at least
\[
 H(A_0)=F(A_0/2)
       -\frac3{k_0}\log\frac{1+\exp(k_0F(A_0)/3)}2.
\]
The function $H$ increases for $A>2$, since
$\tfrac12F'(A/2)>F'(A)$ and the derivative of
$\log(1+e^t)$ is less than one.  Thus $A_0>22/5$, and
$\log((1+e^t)/2)\le t/2+t^2/8$ gives
\[
 H(A_0)>F(11/5)-\tfrac12F(22/5)
                   -\frac{F(22/5)^2}{48\log^22}>\delta.
\]
The final inequality is in \eqref{tri:constant-table}.
\end{proof}

\begin{lemma}\label{tri:two-class}
Under \eqref{tri:range}, if $c=0$, then
$\mathsf N_0^{3b}(S)\ge B(\ell,3b)$.
\end{lemma}
\begin{proof}
If $d<3b$, then $L=\lfloor a/3\rfloor\ge4b$ and
$a/\lceil\ell/2\rceil\ge3/2$.  Every factor of
$\binom{\lfloor L/2\rfloor}b/\binom Lb$ is at least $1/3$.
Thus \eqref{tri:one-class} divided by $B(\ell,3b)$ is at least
$((3/2)^3/3)^b=(9/8)^b>1$.

Suppose $d\ge3b$, and put $T=\binom a{3b}/\binom d{3b}$.
First suppose $T\le\binom{\lfloor a/3\rfloor}b$.  Choose
$x\ge b$ with $\binom xb=T$, round it as in
Lemma~\ref{tri:rounding}, and use $U$ triples from the first class
and $b$ triples from the second.  Lemma~\ref{tri:hybrid} gives
\begin{equation}\label{tri:two-capped-count}
 \mathsf N_0^{3b}(S)\ge\binom d{3b}\mathbb E B(U+b,b).
\end{equation}
The cap implies $d>6b$.  Otherwise $A=a/(3b)\ge24/5>43/10$,
and convexity of $\log((a-s)/(d-s))$ gives
\[
 T\ge\left(\frac{A-1/2}{5/3}\right)^{3b}
       >(eA)^b\ge\binom{\lfloor a/3\rfloor}b,
\]
a contradiction.  The strict inequality follows already at
$A=43/10$ from $(57/25)^3>(43/10)(11/4)$ and $e<11/4$,
and its ratio increases with $A$.
If $T\le b+1$, then $\mathbb E B(U+b,b)=T+1$, and
\eqref{tri:two-capped-count} is at least
$\binom a{3b}+\binom d{3b}\ge B(\ell,3b)$.
Otherwise Lemmas~\ref{tri:rounding} and \ref{tri:midpoint} give
\[
 \mathsf N_0^{3b}(S)\ge2\binom{\ell/2}{3b}e^{\Delta_b}.
\]
For odd $\ell$, \eqref{tri:parity} costs at most
$\tfrac18\log^2(19/13)<361/20000<\Delta_b$; for even $\ell$
there is no loss.

It remains to treat $T>\binom Lb$.  Write
$N=\ell/(6b)$, $A=a/(3b)$, $D=d/(3b)$, and $t=L/b$.
Put
\[
 f(A)=2F(A)-3F(2N-A),\quad
 J(A)=F(A/2)+2F(A)-3F(N).
\]
On $N\le A<2N-1$, $0<J'\le f'$.  For the second inequality
use $D\le A$, $A\ge3$, and
$\tfrac12F'(A/2)\le1/(A-2)\le3/A\le3F'(D)$.
All ensuing comparisons extend to $D=1$ by continuity.
By \eqref{tri:entropy-ratio}, \eqref{tri:integer-lower}, and
$A-1/b\le t\le A$, cap failure gives
\[
 bf(A)>-\eta,\qquad
 \eta=\tfrac12\log(2\pi b)+\tfrac12F'(t)
                 +\frac{1+1/(t-1)}{12b}.
\]
Comparing with the root in Lemma~\ref{tri:entropy-root} yields
$bJ(A)>b\delta-\eta$.

Monotonicity of $R_b$ and convexity give
\[
 \frac{B(L,b)}{\binom Lb}
 \ge2\exp\{b[F(t/2)-F(t)]\}.
\]
With $Q=\lceil\ell/2\rceil$, the factorial bounds give
\begin{equation}\label{tri:original-ratio}
 \log\frac{\binom a{3b}}{\binom Q{3b}}
 \ge3b[F(A)-F(N)]-F'(N)-\frac{1+1/N}{36b}.
\end{equation}
Indeed the ratio of normalized factors loses at most
$\frac12F'(N)+(1+1/N)/(36b)$, and replacing $\ell/2$ by $Q$
loses at most $F'(N)/2$.
Let $D_*(x)=F(x/2)-F(x)/2$.  Its derivative is positive and
decreasing for $x>2$.  Combining the preceding estimates with
\eqref{tri:one-class}, and using $t\ge17/5-1/b$, gives
\begin{align}
 \log\frac{B(L,b)\binom a{3b}}{\binom Lb B(\ell,3b)}
 >{}&b\delta-\tfrac12\log(2\pi b)-D_*'(17/5-1/b)-F'(17/5)
 \notag\\[-2pt]
 &-\frac{1+1/(17/5-1/b-1)}{12b}
       -\frac{1+5/17}{36b}.\label{tri:two-loss}
\end{align}
The right side increases for $b\ge35$.  At $b=35$ it is greater
than $19/1000$, by \eqref{tri:log-certificate} with $\pi<22/7$.
This completes the proof.
\end{proof}

\begin{lemma}\label{tri:large-class}
Under \eqref{tri:range}, if $a\ge16\ell/25$, then
$\mathsf N_0^{3b}(S)\ge B(\ell,3b)$.
\end{lemma}
\begin{proof}
Put $N=\ell/(6b)$, $A=a/(3b)$, $L=\lfloor a/3\rfloor$,
$t=L/b$, and $E(x)=F(x/2)-F(x)$.  The proof of
\eqref{tri:two-loss}, without the cap estimate, gives
\[
 \log\frac{B(L,b)\binom a{3b}}{\binom LbB(\ell,3b)}
 \ge bJ(A,N)-E'(t)-F'(N)-\frac{1+1/N}{36b},
\]
where $J(A,N)=F(A/2)+2F(A)-3F(N)$.
For $p=16/25$, this increases in $A\ge2pN$, while
\[
 J(2pN,N)=\log(4p^3)
 -\sum_{r\ge1}\frac{p^{-r}+2(2p)^{-r}-3}{r(r+1)N^r}
\]
increases in $N$.  Every coefficient subtracted is positive, by
convexity of the power function and $2/(3p)>1$.  Consequently
$J(A,N)>27/2000$, by \eqref{tri:constant-table}.
As $t\ge544/125-1/b$ and $E',F'$ decrease, the displayed lower
bound is at least
\[
 \frac{27b}{2000}-E'(544/125-1/b)-F'(17/5)
       -\frac{1+5/17}{36b}.
\]
It increases with $b\ge35$ and at $35$ is greater than $3/40$,
by \eqref{tri:log-certificate}.  Now apply
\eqref{tri:one-class}.
\end{proof}

\subsection{Profiles without a majority class}

\begin{lemma}\label{tri:thin-balanced}
Under \eqref{tri:range}, suppose $c>0$, $a\le d+c$, and
$c\le d/4$.  Then $\mathsf N_0^{3b}(S)\ge B(\ell,3b)$.
\end{lemma}
\begin{proof}
Put $\mu=(b-1)/2$ and choose
\[
 u=\max\left\{b,\left\lceil\mu+
 \frac{(b+1)(d-2b)(d-2b+1)}{12b(d-\mu)}\right\rceil\right\}.
\]
Use $c$ transversal triples and $u$ triples in each of the two
larger classes.  We verify feasibility and the numerical count.
Set $X=d/b$, $z=c/b$, and $W=(2u-b+1)/b$.  Since
$\ell\le2d+2c\le5d/2$, we have $X\ge X_0:=204/25$.
The ceiling error and $1/b\le1/35$ give
\begin{equation}\label{tri:W-thin}
 W\le\max\{36/35,w(X)\},\qquad
 w(X)=\frac2{35}+\frac6{35}
                  \frac{(X-2)(X-69/35)}{X-1/2}.
\end{equation}
For $X\ge X_0$,
\begin{equation}\label{tri:W-compare}
 \max\{36/35,w(X)\}<
 \min\left\{\frac{(X-3)^3}
 {3(X-3)^2+(3X/4)(X-3)+X^2/16},\frac X2-1\right\}.
\end{equation}
Here are details of the rational comparisons.  Write the denominator
of the first fraction as $A(X)$ and that fraction as $L(X)$.
It increases, since
\[
 L(X)=\frac{X-3}{3+(3/4)X/(X-3)+(1/16)(X/(X-3))^2},
\]
and $L(X_0)-36/35=32061/200725$.  After positive denominators
are cleared, $L(X)>w(X)$ is the positivity of
\[
 (X-3)^3(X-1/2)
 -A(X)\{(2/35)(X-1/2)+(6/35)(X-2)(X-69/35)\}.
\]
Its coefficients in ascending powers of $X-X_0$ are
\[
 \frac{961754409}{3906250},\quad
 \frac{448955211}{2187500},\quad
 \frac{148972183}{2450000},\quad
 \frac{187587}{24500},\quad\frac{97}{280}.
\]
For $X/2-1>w(X)$ the corresponding coefficients are
$727123/43750$, $116031/24500$, and $23/70$.  The comparison
with the constant $36/35$ already holds at $X_0$.

Equations \eqref{tri:W-thin}--\eqref{tri:W-compare} imply
$u<d/4-1/2$, whence $3u+c<d$.  The type is feasible.
By \eqref{tri:midpoint-product} and the definition of $u$,
\[
 \frac{\binom ub}{\binom d{3b}}\ge\kappa^bR.
\]
This is the largest endpoint, so Lemma~\ref{tri:hybrid} gives
\begin{equation}\label{tri:thin-balanced-count}
 \mathsf N_0^{3b}(S)\ge
 B(2u+c,b)\frac{\binom d{3b}}{\binom ub}.
\end{equation}
Convexity of $t\mapsto\log(1+s/t)$ and Jensen's inequality give
\[
 \frac{\binom{u+c/2}b}{\binom ub}\ge(1+z/W)^b,
 \qquad
 \frac{\binom{d+c}{3b}}{\binom d{3b}}
 \le\left(1+\frac z{X-3+1/b}\right)^{3b}.
\]
For $H>0$, the inequality $1+z/W\ge(1+z/H)^3$ is equivalent
to $W\le H^3/(3H^2+3zH+z^2)$.  The last expression increases
in $H$ and decreases in $z$.  Since $H=X-3+1/b\ge X-3$ and
$z\le X/4$, \eqref{tri:W-compare} proves
\[
 \frac{\binom{u+c/2}b}{\binom ub}
 \ge\frac{\binom{d+c}{3b}}{\binom d{3b}}.
\]
Finally $B(2u+c,b)\ge2\binom{u+c/2}b$, while
$\lceil\ell/2\rceil\le d+c$ gives
$B(\ell,3b)\le2\binom{d+c}{3b}$.  Substitute these inequalities
in \eqref{tri:thin-balanced-count}.
\end{proof}

\begin{lemma}\label{tri:dense-balanced}
Under \eqref{tri:range}, suppose $a\le d+c$ and $c\ge d/4$.
Then the lower bound in \eqref{tri:rainbow} exceeds $B(\ell,3b)$.
\end{lemma}
\begin{proof}
Write $s=\ell/b$, $A=a/b$, $D=d/b$, and $C=c/b$.
We first show
\begin{equation}\label{tri:dense-gap}
 J:=F(C/2)+F(A)+F(D)-3F(s/6)>1/5.
\end{equation}
At fixed $s,C$, concavity of $F$ shows that $F(A)+F(D)$ is
smallest for the largest feasible $A$.  Since $A\le s/2$ and
$D\le4C$, we have $s/10\le C\le s/3$.  The extreme pairs are
$(A,D)=(s/2,s/2-C)$ for $s/10\le C\le s/4$, and
$(A,D)=(s-2C,C)$ for $s/4\le C\le s/3$.  The resulting
functions of $C$ are concave.  Thus it suffices to consider
\[
 (A,D,C)=(s/2,2s/5,s/10),\quad(s/2,s/4,s/4),
                 \quad(s/3,s/3,s/3).
\]
Their entropy gaps are
\[
 \begin{split}
 E_1&=F(s/20)+F(s/2)+F(2s/5)-3F(s/6),\\
 E_2&=F(s/8)+F(s/2)+F(s/4)-3F(s/6),\\
 E_3&=2F(s/3)-2F(s/6).
 \end{split}
\]
Use $F(x)=1+\int_0^1\log(x-t)\,dt$.
On putting $x=s/20$, the integrand comparison $E_2>E_1$ follows
from
\[
 (5x/2-t)(5x-t)-(x-t)(8x-t)=\tfrac92x^2+\tfrac32xt>0.
\]
Similarly, $E_3>E_1$ follows from
\[
 (s/6-t)(s/3-t)^2-(s/20-t)(s/2-t)(2s/5-t)
 =\frac{23s^3}{2700}+\frac{41s^2t}{1800}-\frac{7st^2}{60}>0,
\]
because $s\ge102/5$ and $0\le t\le1$.
The convergent expansion
\begin{equation}\label{tri:F-series}
 F(x)=\log x+1-\sum_{r\ge1}\frac1{r(r+1)x^r}
\end{equation}
shows that $E_1$ increases with $s$: each coefficient
$20^r+2^r+(5/2)^r-3\cdot6^r$ is positive.  At $s=102/5$,
the required lower bound is in \eqref{tri:constant-table}.

To pass from entropy to the exact count, put $Q=\lceil\ell/2\rceil$.
Use $B(c,b)\ge2\binom{c/2}b$, apply
\eqref{tri:real-lower} to the three factors of
\eqref{tri:rainbow}, and apply \eqref{tri:integer-upper} to
$B(\ell,3b)\le2\binom Q{3b}$.  Since
$Q/(Q-3b)\le17/12$ and the entropy increment from $\ell/2$ to
$Q$ is at most $\frac12\log(17/12)$, the result is
\[
 \log\frac{B(c,b)\binom ab\binom db}{B(\ell,3b)}
 \ge bJ-\log b-\log\frac{2\pi}{\sqrt3}
       -\log\frac{17}{12}-\frac1{4b}-\frac1{12Q}
 >\frac b5-\log b-2-\frac1b>0.
\]
Here $(2\pi/\sqrt3)(17/12)<187/35<e^2$, and the last
expression increases for $b\ge35$ and is positive at $35$.
\end{proof}

\subsection{A majority class and a relatively small third class}

In this subsection the assumptions are
\begin{equation}\label{tri:thin-majority}
 d+c\le a\le\frac{16}{9}(d+c),\qquad 1\le c\le d/4,
 \qquad \ell\ge102b/5.
\end{equation}
In particular
\begin{equation}\label{tri:D0}
 d\ge D_0b,\qquad D_0=3672/625,
\end{equation}
since $\ell\le25(d+c)/9\le125d/36$.
Let
\[
 u_{\max}=\left\lfloor\frac{a-c}{3}\right\rfloor,
 \qquad T=\frac{\binom a{3b}}{\binom d{3b}},\qquad
 \alpha_b=\binom d{3b}^{-1}.
\]
There are three possibilities: $T>\binom{u_{\max}}b$;
$T\le\binom{u_{\max}}b$ and $\alpha_b\ge\kappa^bR$;
or $\alpha_b<\kappa^bR$.  We treat them in that order.

\begin{lemma}\label{tri:one-sided-endpoint}
Under \eqref{tri:thin-majority}, the type
$(r,u,v)=(c,u_{\max},0)$ is feasible and
\begin{equation}\label{tri:one-sided-eligible}
 \frac{\binom{u_{\max}}b}{\binom a{3b}}\ge\kappa^bR.
\end{equation}
Consequently
\begin{equation}\label{tri:one-sided-count}
 \mathsf N_0^{3b}(S)\ge
       B(u_{\max}+c,b)\frac{\binom a{3b}}{\binom{u_{\max}}b}.
\end{equation}
\end{lemma}
\begin{proof}
Clearly $u_{\max}\ge b$ and the type is feasible.
The reversed version of \eqref{tri:midpoint-product} is valid here.
Indeed $d>5b$ and $u_{\max}\ge(d-2)/3\ge3b/2$;
in its logarithmic second derivative,
\[
 (b-i)^{-2}-(u_{\max}-i)^{-2}\ge5/(9b^2),\qquad
 (d-i)^{-2}<1/(16b^2).
\]
Thus Jensen's inequality applies even before the two remaining
positive terms are included.  Write $A=a/b$, $D=d/b$, $C=c/b$.
Using $u_{\max}\ge(a-c-2)/3$ and $b/(b+1)\ge35/36$ gives
\[
 \frac{\binom{u_{\max}}b/\binom a{3b}}{\kappa^bR}
 \ge\left\{\frac{(35/18)(2A-2C-106/35)(D-1/2)}
 {(A-2)(A-69/35)}\right\}^{b}.
\]
The base exceeds one.  To verify this, subtract its denominator
from its numerator and call the resulting polynomial $P(A,C,D)$.
It is concave in $A$, so it suffices to take $A=D+C$ or
$A=16(D+C)/9$.  On the first edge it decreases in $C$, and on
the second it is concave in $C$.  Hence only
\[
 (A/D,C/D)=(5/4,1/4),\quad(16/9,0),\quad(20/9,1/4)
\]
remain.  In ascending powers of $D-D_0$, the corresponding
coefficient triples are
\[
 \begin{array}{c|ccc}
 (5/4,1/4)&614718487/9843750&256903/10500&335/144\\
 (16/9,0)&809462633/7031250&14818157/354375&304/81\\
 (20/9,1/4)&866335657/9843750&17689057/567000&295/108.
 \end{array}
\]
All are positive.  This proves \eqref{tri:one-sided-eligible};
Lemma~\ref{tri:hybrid} then gives \eqref{tri:one-sided-count}.
\end{proof}

\begin{lemma}\label{tri:cap-failure}
Under \eqref{tri:thin-majority}, if $T>\binom{u_{\max}}b$,
then $\mathsf N_0^{3b}(S)\ge B(\ell,3b)$.
\end{lemma}
\begin{proof}
Put $N=\ell/(6b)$, $A=a/(3b)$, $C=c/b$, $D=d/(3b)$,
$U=A-C/3$, $V=A/2+C/3$, and $t=u_{\max}/b$.
Then
\[
 A\ge N\ge17/5,\quad D=2N-A-C/3>1,\quad
 0\le C\le3N/5,\quad U\ge4N/5,\quad
 U-\frac2{3b}\le t\le U.
\]
Also $c\le u_{\max}$, since $a\ge d+c\ge5c$.
For fixed $N,C$, define
\[
 f(A,C)=3F(A)-3F(D)-F(U),\quad
 J(A,C)=F(V)-F(U)+3F(A)-3F(N).
\]
The equation $f(A,C)=0$ has a unique root $A_0(C)>N$.
At $A=N$,
$f\le3F(N)-4F(4N/5)<0$: this last expression decreases
with $N$ and is negative at $17/5$ by
\eqref{tri:constant-table}.  At $D=1$ it is positive.
To see strict monotonicity, write
$p=F'(A)$, $q=F'(D)$, $r=F'(U)$, $s=F'(V)$.
Since $U\ge4A/5$ and $A\ge17/5$, the integral formula for $F'$
gives $r\le2p$.  Consequently $f_A=3p+3q-r>0$ and
$J_A=s/2-r+3p>0$.  Moreover $J_A\le f_A$, because
$6V-D\ge2N>5$ implies
$1/[2(V-z)]\le3/(D-z)$ for $0\le z\le1$.

We claim that $J(A_0(C),C)$ is nondecreasing in $C$.
Implicit differentiation and multiplication by the positive $f_A$
give the derivative the sign of
\begin{equation}\label{tri:root-C}
 sp-sr/2+q(s/2+2r-3p).
\end{equation}
The coefficient of $q$ is nonnegative.  Indeed put
$z=C/(3A)\in[0,1/5]$ and $h=v/A\in[0,5/17]$ in the
integral formula for $F'$.  Clearing the positive denominator of
\[
 \frac1{2(1/2+z-h)}+\frac2{1-z-h}-\frac3{1-h}
\]
leaves $3z^2+(1/2-3z/2)h-h^2/2\ge0$.
Since $q\ge r\ge p$ and $s\ge p$, \eqref{tri:root-C} is at least
$sp+2r^2-3pr\ge(r-p)(2r-p)\ge0$.
At $C=0$, Lemma~\ref{tri:entropy-root} applies.  Therefore
\begin{equation}\label{tri:root-C-gap}
 J(A_0(C),C)>\delta=191/2000.
\end{equation}

Cap failure, \eqref{tri:entropy-ratio}, and
\eqref{tri:integer-lower} imply
\[
 bf(A,C)>-\eta,\qquad
 \eta=\tfrac12\log(2\pi b)+\tfrac16F'(t)
                    +\frac{1+1/(t-1)}{12b}.
\]
Here $b[F(U)-F(t)]\le(2/3)F'(t)$; the positive
$F'(t)/2$ term in \eqref{tri:integer-lower} explains the
coefficient $1/6$.  As $0<J_A\le f_A$, comparison with
\eqref{tri:root-C-gap} gives $bJ(A,C)>b\delta-\eta$.

Monotonicity of $R_b$, convexity, and $c\le u_{\max}$ give
\[
 \frac{B(u_{\max}+c,b)}{\binom{u_{\max}}b}
 \ge2\exp\{b[F((t+C)/2)-F(t)]\}.
\]
The derivative of $F((x+C)/2)-F(x)$ is at most
$E'(x)=\frac12F'(x/2)-F'(x)$, which is positive and decreasing
for $x>2$.  Thus replacing $t$ by $U$ loses at most
$(2/3)E'(t)$.  Combining with \eqref{tri:original-ratio} and
Lemma~\ref{tri:one-sided-endpoint}, we obtain
\begin{align*}
 \log\frac{B(u_{\max}+c,b)\binom a{3b}}
 {\binom{u_{\max}}bB(\ell,3b)}
 >{}&b\delta-\tfrac12\log(2\pi b)-T_*(t)-F'(N)\\
 &-\frac{1+1/(t-1)}{12b}-\frac{1+1/N}{36b},
\end{align*}
where $T_*(x)=\frac13F'(x/2)-\frac12F'(x)$.
This function is positive and decreasing for $x>2$, since
$T_*'(x)=-(x+2)/(6x(x-1)(x-2))$.
Use $t\ge68/25-2/(3b)$ and $N\ge17/5$.
The resulting lower bound increases for $b\ge35$; at $35$
it is greater than $7/100$, by \eqref{tri:log-certificate}
with $\pi<22/7$.  This proves the assertion.
\end{proof}

\begin{lemma}\label{tri:capped}
Under \eqref{tri:thin-majority}, suppose
$T\le\binom{u_{\max}}b$ and $\alpha_b\ge\kappa^bR$.
Then $\mathsf N_0^{3b}(S)\ge B(\ell,3b)$.
\end{lemma}
\begin{proof}
Choose $x\ge b$ with $\binom xb=T$ and round it to $U$ as in
Lemma~\ref{tri:rounding}.  Use $c$ transversal triples, $U$
triples from the class of size $a$, and $b$ triples from the
class of size $d$.  The cap gives $U\le u_{\max}$, and
$d-c\ge3d/4>3b$ by \eqref{tri:D0}.  The expected two
monochromatic endpoints are both $\alpha_b$.  Hence
\begin{equation}\label{tri:capped-count}
 \mathsf N_0^{3b}(S)\ge\binom d{3b}\mathbb E B(U+b+c,b).
\end{equation}

First suppose $x\le b+c$.  For integers $b\le u\le w$, the
function $B(u+w,b)-\binom ub$ is nondecreasing in $u$, since its
increment is $\binom{\lfloor(u+w)/2\rfloor}{b-1}-\binom u{b-1}$.
Taking $w=b+c$ gives
\[
 \mathbb E B(U+b+c,b)\ge T+K,
 \qquad K=B(2b+c,b)-1.
\]
We show
\begin{equation}\label{tri:small-growth}
 K\ge\frac{\binom{d+c}{3b}}{\binom d{3b}}.
\end{equation}
Convexity gives $K\ge\binom{b+c/2}b$, whereas
\[
 \log\frac{\binom{d+c}{3b}}{\binom d{3b}}
 \le c\log\frac d{d-3b}.
\]
Put $C=c/b$, $D=d/b$, $\eta=144/1225$, and
\[
 L(C)=\frac14\sum_{r=1}^4
       \log\left(1+\frac{4C}{2r-1+\eta}\right).
\]
Then $\log\binom{b+c/2}b\ge bL(C)$.  To verify this, represent
the uniform distribution on $1,\ldots,b$ by $J(t)=\lceil t\rceil$
on $0<t<b$, and split that interval into four equal intervals.
The mean of $J$ on interval $r$ is at most
$(2r-1)b/8+1/2+1/(2b)$.  Indeed the primitive of
$\lceil t\rceil-t-1/2$ lies between $0$ and $1/8$.
Apply Jensen's inequality to the convex decreasing function
$z\mapsto\log(1+c/(2z))$ on each interval, and then use
$4/b+4/b^2\le\eta$.

Set $C_0=D_0/4=918/625$.  Formula
\eqref{tri:log-certificate} gives the exact inequality
\[
 L(C_0)-C_0\log\frac{D_0}{D_0-3}>\frac3{200}.
\]
If $C\le C_0$, each function $\log(1+kC)/C$ decreases, so
$L(C)/C\ge L(C_0)/C_0>\log(D/(D-3))$.
If $C\ge C_0$, use $D\ge4C$ and the fact that
$C\log(4C/(4C-3))$ decreases for $C>3/4$.  Again
$L(C)>C\log(D/(D-3))$.  This proves
\eqref{tri:small-growth}.  Equation \eqref{tri:capped-count}
now yields
\[
 \mathsf N_0^{3b}(S)\ge\binom a{3b}+\binom{d+c}{3b}
 \ge B(\ell,3b).
\]

It remains that $x\ge b+c$.  The cap forces $d>6b$.
For if $d\le6b$, then $c\le d/4$ and \eqref{tri:range}
give $a\ge129b/10$, so $A=a/(3b)\ge43/10$.  The same
Jensen estimate used in Lemma~\ref{tri:two-class} contradicts
$T\le\binom{\lfloor a/3\rfloor}b$.
If $x=b+1$, then $c=1$, already covered by the first case.
We may therefore apply Lemma~\ref{tri:midpoint}.
For $0\le s\le c$, put
$X_s=(x+b+s)/2$ and $Y_s=(a+d+s)/2$.
Since $3x+c\le a$ and $d\ge3b+c+4$,
$X_s\le(Y_s-2)/3$.  The latter inequality for $d$ follows
from $d-c\ge3D_0b/4>3b+4$.
Grouping the reciprocal terms in threes gives
\[
 H_b(X_s)\ge\sum_{i=0}^{b-1}\frac3{Y_s-3i-2}
             \ge H_{3b}(Y_s).
\]
Thus $\log\binom{X_s}b-\log(\binom{Y_s}{3b}/\binom d{3b})$
is nondecreasing in $s$.  Lemmas~\ref{tri:rounding} and
\ref{tri:midpoint}, followed by \eqref{tri:capped-count}, give
\[
 \mathsf N_0^{3b}(S)\ge2\binom{\ell/2}{3b}e^{\Delta_b}.
\]
As in Lemma~\ref{tri:two-class}, $\Delta_b$ is larger than
the parity correction in \eqref{tri:parity}.
\end{proof}

The last thin-class case requires a larger number of triples in
the second class.  We isolate the rational estimates that ensure
these triples and the corresponding triples in the first class
fit inside their residue classes.

\begin{lemma}\label{tri:rational-cap}
Let $D>9$ and
\[
 D\le A\le\min\{T_0(D),20D/9\},\qquad
 T_0(D)=\frac{35D^2-139D+240}{210}.
\]
Put
\[
 K=37/70,\quad L=26/35,\quad
 H=\frac3{35}(D-2)(D-69/35),\quad
 W=\frac12\left(\frac AD+\frac{A-3}{D-3}\right),
\]
and
\[
 g_0=\frac3A(K+LW^3),\qquad
 g_1=\frac3A\left\{K+\left(\frac1{35}+\frac H{A-1/2}\right)W^3\right\}.
\]
Then $g_0\le41/50$ and $g_1<2/3$.  If also
$A\le16(D+1/35)/9$, then $g_0\le2/3$.
\end{lemma}
\begin{proof}
For fixed $D$, write $W=\alpha A-\beta$ with $\alpha>0$ and
$\beta>0$.  Differentiation shows that $g_0$ increases in $A$:
the derivative, multiplied by $A^2/3$, is
$-K+LW^2(2\alpha A+\beta)>-K+2L>0$.
For a polynomial upper bound $A=A(D)$, the inequality
$g_0\le\gamma$ is equivalent to the nonnegativity of
\begin{equation}\label{tri:cap-polynomial}
 P_{A,\gamma}(D)=
 [\gamma A(D)-3K]\{2D(D-3)\}^3
 -3L\{(2D-3)A(D)-3D\}^3.
\end{equation}
For completeness, the exact polynomial checks are given here.
For a polynomial $p$, its Bernstein coefficients on $[l,h]$ are
\begin{equation}\label{tri:Bernstein}
 \beta_i=\sum_{j=0}^i a_j\frac{\binom ij}{\binom kj},
 \quad p(l+(h-l)z)=\sum_{j=0}^k a_jz^j.
\end{equation}
Then $p(l+(h-l)z)=\sum_i\beta_i\binom ki z^i(1-z)^{k-i}$,
so nonnegative coefficients imply nonnegativity throughout the
interval.  Substitution in \eqref{tri:cap-polynomial} gives
\[
 \begin{array}{c|c|c|c}
 A(D)&\gamma&[l,h]&\min_i\beta_i\\ \hline
 T_0(D)&41/50&[9,169/10]&1356247970262/300125\\
 T_0(D)&2/3&[9,71/5]&34109173062/12005.
 \end{array}
\]
On the two remaining intervals, the coefficients in ascending
powers of $D-h$ are respectively
\[
 \begin{split}
 &\left(\frac{149341385233033487}{10631250000},
 \frac{40922548069229021}{177187500},
 \frac{3076163845099189}{35437500},
 \frac{2081329389739}{151875},\right.\\[-2pt]
 &\hspace{25mm}\left.
 \frac{11698351289}{10125},\frac{3883258532}{70875},
 \frac{58897796}{42525},\frac{656}{45}\right)
 \end{split}
\]
for $A(D)=20D/9$, $\gamma=41/50$, $h=169/10$, and
\[
 \begin{split}
 &\left(\frac{17787734999908802}{1688203125},
 \frac{51212120643126526}{1012921875},
 \frac{4480446483104338}{202584375},
 \frac{60972256639606}{14586075},\right.\\[-2pt]
 &\hspace{25mm}\left.
 \frac{1475536698092}{3472875},\frac{483213884}{19845},
 \frac{1266460}{1701},\frac{256}{27}\right)
 \end{split}
\]
for $A(D)=16(D+1/35)/9$, $\gamma=2/3$, $h=71/5$.
All displayed coefficients are positive, proving the assertions
about $g_0$.

For $g_1$, suppose first $D\ge10$ and consider
\[
 P(A)=\frac{2A}{9}-K-\left(\frac1{35}+\frac H{A-1/2}\right)W^3.
\]
Write $W=k(A-z)$, where
$k=(2D-3)/(2D(D-3))$ and $z=3D/(2D-3)>1/2$.
Both $(A-z)^3$ and $(A-z)^3/(A-1/2)$ are convex for $A\ge D$.
For the latter assertion, set $B=A-1/2$, $h=z-1/2$; the
second derivative of $B^2-3hB+3h^2-h^3/B$ is
$2(1-(h/B)^3)>0$.  Thus $P$ is concave and it suffices to check
$A=D$ and $A=20D/9$.  At $A=D$, the coefficients of
$(D-1/2)P(D)$ in powers of $D-10$ are
\[
 454801/44100,\qquad17657/7350,\qquad43/315.
\]
At $A=20D/9$, multiply $P$ by
$(A-1/2)18^3(D-3)^3$; its coefficients in powers of $D-10$ are
\[
 \frac{83025201271}{22050},\quad\frac{78124012991}{11025},\quad
 \frac{2053467851}{735},\quad\frac{1001283838}{2205},\quad
 \frac{14703344}{441},\quad\frac{6400}{7}.
\]
They are positive, so $P>0$.

Finally, if $9<D<10$, the inequalities $D\le A\le T_0(D)$
force $D>46/5$ and $A<235/21$.  The function $H$ increases
with $D$, while $W$ increases with $A$ and decreases with $D$.
Put
\[
 H_+=\frac3{35}\,8(10-69/35),\qquad
 W_+=\frac12\left(\frac{235/21}{46/5}
                         +\frac{235/21-3}{46/5-3}\right).
\]
Then
\[
 g_1\le\frac3{46/5}\left\{K+
 \left(\frac1{35}+\frac{H_+}{46/5-1/2}\right)W_+^3\right\}
 =\frac{8362514769036013}{13647181975201152}<\frac23.
\]
This proves the remaining assertion.
\end{proof}

\begin{lemma}\label{tri:low-alpha}
Under \eqref{tri:thin-majority}, if $\alpha_b<\kappa^bR$, then
$\mathsf N_0^{3b}(S)\ge B(\ell,3b)$.
\end{lemma}
\begin{proof}
Put $\mu=(b-1)/2$ and choose
\[
 v=\max\left\{\left\lceil\frac{6b}{5}\right\rceil,
 \left\lceil\mu+
 \frac{(b+1)(d-2b)(d-2b+1)}{12b(a-\mu)}\right\rceil\right\}.
\]
Choose $x\ge v$ so that
\begin{equation}\label{tri:large-baseline-inverse}
 \frac{\binom xb}{\binom vb}
       =\frac{\binom a{3b}}{\binom d{3b}}.
\end{equation}
Set $A=a/b$, $D=d/b$, $t=1/b$.
Applying \eqref{tri:midpoint-product} at $u=b$, the hypothesis
$\alpha_b<\kappa^bR$ gives
\[
 6(A-1/2+t/2)<(D-2)(D-2+t),
 \qquad A<T_0(D).
\]
It also gives $D>9$: $D>D_0>5$, whereas the convex quadratic
$T_0(D)-D$ is negative at both $5$ and $9$.
The profile gives $A\le20D/9$; if $c=1$, it gives
$A\le16(D+1/35)/9$.

With $W$ as in Lemma~\ref{tri:rational-cap},
\begin{equation}\label{tri:centered-inverse}
 x\le\mu+(v-\mu)W^3.
\end{equation}
Indeed Jensen's inequality applied to
$i\mapsto\log((x-i)/(v-i))$ bounds the left ratio of
\eqref{tri:large-baseline-inverse} below by
$((x-\mu)/(v-\mu))^b$.  On the other side, the increasing
convex function $s\mapsto\log((a-s)/(d-s))$ has its sum over
$0,\ldots,3b-1$ bounded above by its integral over $[0,3b]$,
and then by the trapezoid estimate.  Hence that ratio is at most
\[
 \left(\frac{a(a-3b)}{d(d-3b)}\right)^{3b/2}\le W^{3b}.
\]
The ceiling errors in the definition of $v$ and
\eqref{tri:centered-inverse} give
\[
 \frac{3(x+1)}a\le\max\{g_0(A,D),g_1(A,D)\}.
\]
More explicitly, $(\mu+1)/b\le37/70$; the bound for
$(v-\mu)/b$ is $26/35$ for the first branch, and
$1/35+H/(A-1/2)$ for the second.  This explains the two
functions in Lemma~\ref{tri:rational-cap}.
Furthermore $3(v+1)/d<1/2$.  In the first branch it is at most
$132/(35D)<1/2$; in the other branch use
\[
 \frac{(D-2)(D-69/35)}{A-1/2}\le D
\]
to bound it by $9/35+117/(70D)<31/70$.
Consequently, with
\begin{equation}\label{tri:gamma}
 \gamma=\begin{cases}2/3,&c=1,\\41/50,&c\ge2,\end{cases}
 \qquad x+1\le\gamma a/3,
 \qquad v+1\le\gamma d/3.
\end{equation}

We must still check the capacities with the $c$ transversal
triples present.  If $a/d\ge25/18$, then
$c/a\le9/50$, and $3\lceil x\rceil+c\le a$ follows from
\eqref{tri:gamma}.  If $a/d\le25/18$, monotonicity of $g_0$
in $A$ and its decrease in $D$ along $A=rD$ give
\[
 g_0(A,D)\le g_0(25/2,9)<3/4.
\]
The last inequality is obtained by substitution in its displayed
rational formula.  Together with $g_1<2/3$, it gives
$3\lceil x\rceil\le3a/4$, while $c/a\le1/5$.
Finally $3v+c<d/2+d/4<d$.  Thus all the required types are
feasible.

Round $x$ to $U$ as in Lemma~\ref{tri:rounding}, and put
$\alpha=\binom vb/\binom d{3b}$.  The choice of $v$ and
\eqref{tri:midpoint-product} ensure $\alpha\ge\kappa^bR$.
The mixture with $c$ transversal triples, $U$ triples from the
first class, and $v$ triples from the second therefore gives
\begin{equation}\label{tri:low-alpha-count}
 \mathsf N_0^{3b}(S)\ge\frac{\mathbb E B(U+v+c,b)}\alpha.
\end{equation}
For this rounding the loss $\varepsilon=x-\mathbb E U$ satisfies
\begin{equation}\label{tri:low-round}
 0\le\varepsilon\le\frac{\log6}{8}<\frac9{40}.
\end{equation}
Indeed $(\log f_b')'=H_b+H_b'/H_b\le H_b$, and on the rounding
interval, whose left endpoint is at least $v\ge6b/5$,
$H_b\le\log(v/(v-b))\le\log6$.  Apply the second assertion
of Lemma~\ref{tri:rounding}.

Define $g(y)$ on $[d,a]$ by
$\binom{g(y)}b=\alpha\binom y{3b}$.
Then $g(d)=v$, $g(a)=x$, and $g$ is convex.  In fact implicit
differentiation gives $g''\ge0$ whenever
$\rho_b(g(y))\ge\rho_{3b}(y)$, which follows from
\eqref{tri:Hr} and $y\ge d>9b$.
Convexity and \eqref{tri:gamma} give
$g(y)+1\le\gamma y/3$.  The integral bounds
\eqref{tri:Hr} now imply
\[
 H_b(g(y))\ge\log\frac{g(y)+1}{g(y)+1-b}
 \ge\log\frac y{y-3b/\gamma}
 \ge\frac1\gamma\log\frac y{y-3b}
 \ge\frac1\gamma H_{3b}(y).
\]
Thus $g'(y)\le\gamma$.

Put $Y=\ell/2$.  Since $a\ge d+c$, both $(a+d)/2$ and $Y$
lie in $[d,a]$.  Hence
\[
 g(Y)\le\frac{x+v}{2}+\frac{\gamma c}{2},\qquad
 \frac{\mathbb E U+v+c}{2}-g(Y)
 \ge\frac{c(1-\gamma)-\varepsilon}{2}>e,
\]
where $e=13/240$ if $c=1$, and $e=27/400$ if $c\ge2$.
Two applications of convexity in \eqref{tri:low-alpha-count}
give
\[
 \mathsf N_0^{3b}(S)\ge
       2\binom Y{3b}\exp\left(\frac{3be}{\gamma Y}\right).
\]
Indeed $e<1$ and $g(Y)+e\le\gamma Y/3$, so the increase in
$\log\binom zb$ from $g(Y)$ to $g(Y)+e$ is at least
$3be/(\gamma Y)$.
For odd $\ell$, \eqref{tri:parity} costs at most
$9b^2/(8(Y-3b)^2)$.  As $Y/b\ge51/5$, it is enough that
$e>85\gamma/1152$.  The two differences are respectively
\[
 \frac{13}{240}-\frac{85(2/3)}{1152}=\frac{43}{8640}>0,
 \qquad
\frac{27}{400}-\frac{85(41/50)}{1152}=\frac{403}{57600}>0.
\]
The even case requires no correction.  This proves the lemma.
\end{proof}

\subsection{A majority class and a relatively large third class}

\begin{lemma}\label{tri:dense-one-sided}
Under \eqref{tri:range}, suppose
\[
 d+c\le a\le16(d+c)/9,\qquad d\le4c,\qquad a\ge4c.
\]
Then $\mathsf N_0^{3b}(S)\ge B(\ell,3b)$.
\end{lemma}
\begin{proof}
Use $c$ transversal triples and $c$ triples from the largest
class.  This is feasible, and
\[
 c\ge918b/625>b,\qquad d\ge459b/125>3b,
\]
by $\ell\le25(d+c)/9\le125c/9$ and
$\ell\le50d/9$.  Put $A=\binom cb/\binom a{3b}$.
Lemma~\ref{tri:hybrid} gives
\[
 \mathsf N_0^{3b}(S)\ge\frac{2\binom cb}{\max\{A,\kappa^bR\}}.
\]
The $A$ endpoint yields $2\binom a{3b}\ge B(\ell,3b)$.
For the other endpoint it suffices, with $Q=\lceil\ell/2\rceil$,
to prove
\begin{equation}\label{tri:dense-product}
 P:=\frac{\binom cb\binom ab\binom db}
               {\kappa^b\binom Q{3b}}>1.
\end{equation}
Factorial cancellation gives
\[
 P=\prod_{i=0}^{b-1}
 \frac{6b(c-i)(a-i)(d-i)}
 {(b-i)(Q-i)(Q-b-i)(Q-2b-i)}.
\]
The logarithm of a factor is convex in $i\in[0,b-1]$.
Indeed the positive $(Q-i)^{-2}$ term compensates the negative
$-(a-i)^{-2}$ term because $a\ge Q$.  Also
$c\ge(918/625)b>\sqrt2b$ and $d\ge c$ imply
\[
 (b-i)^{-2}\ge(c-i)^{-2}+(d-i)^{-2}.
\]
The two remaining terms in the second derivative are positive.
Jensen's inequality at $\mu=(b-1)/2$ therefore gives
\[
 P\ge\left\{
 \frac{12b(c-\mu)(a-\mu)(d-\mu)}
 {(b+1)(Q-\mu)(Q-b-\mu)(Q-2b-\mu)}\right\}^{b}.
\]
Write $s=\ell/b$, $p=a/\ell$, $q=d/\ell$, $r=c/\ell$.
Using $b/(b+1)\ge35/36$, $Q\le\ell/2+1/2$, and
$1/b\le1/35$, the base is at least
\begin{equation}\label{tri:dense-base}
 \frac{(35/3)(ps-1/2)(qs-1/2)(rs-1/2)}
 {(s/2-1/2+1/35)(s/2-3/2+1/35)(s/2-5/2+1/35)}.
\end{equation}
The feasible density region is the quadrilateral
\[
 \tfrac12\le p\le\tfrac{16}{25},\qquad
 (1-p)/5\le r\le p/4,
 \qquad q=1-p-r.
\]
The logarithm of the numerator product in
\eqref{tri:dense-base} is concave.  Its minimum is attained at
a vertex.  At the four vertices, subtract the denominator of
\eqref{tri:dense-base} from its numerator.  The coefficients in
ascending powers of $s-102/5$ are as follows:
\[
\begin{array}{c|cccc}
(p,q,r)&1&s-102/5&(s-102/5)^2&(s-102/5)^3\\ \hline
(1/2,2/5,1/10)
 &\frac{1745893481}{2572500}&\frac{16379729}{147000}
 &\frac{3172}{525}&\frac{13}{120}\\[3pt]
(1/2,3/8,1/8)
 &\frac{4126112783}{4116000}&\frac{1174969}{7350}
 &\frac{113651}{13440}&\frac{19}{128}\\[3pt]
(16/25,36/125,9/125)
 &\frac{4282606768163}{40195312500}&\frac{52058718077}{2296875000}
 &\frac{32027931}{21875000}&\frac{18643}{625000}\\[3pt]
(16/25,1/5,4/25)
 &\frac{255069701023}{321562500}&\frac{2296636417}{18375000}
 &\frac{3436453}{525000}&\frac{1709}{15000}.
\end{array}
\]
Every coefficient is positive.  Thus \eqref{tri:dense-base}
exceeds one, proving \eqref{tri:dense-product} and the lemma.
\end{proof}

\begin{lemma}\label{tri:comparable-majority}
Under \eqref{tri:range}, suppose $d+c\le a\le4c$.
Then the bound in \eqref{tri:rainbow} exceeds $B(\ell,3b)$.
\end{lemma}
\begin{proof}
Since $d\le3c$ and $\ell\le8c$, we have $C=c/b\ge51/20$.
Put $r=d/c$, $t=a/c$.  Then $1\le r\le3$ and
$r+1\le t\le4$.  Consider
\[
 J=F(C/2)+F(rC)+F(tC)-3F((t+r+1)C/6).
\]
At fixed $C,r$, this decreases in $t$.  Indeed, with
$N=(t+r+1)C/6$, we have $tC\ge2N$, and
$F'(2N)<F'(N)/2$ follows from the positive series for $F'$.
We may therefore set $t=4$.
At this value the sign of $J_r$ is that of
\[
 P(r)=2rC(5-2r)-(r+5).
\]
This follows by exponentiating twice the derivative and clearing
positive denominators.  The polynomial is concave,
$P(1)=6C-6>0$, and $P(3)=-6C-8<0$.  Hence $J$ first
increases and then decreases, and its minimum is at $r=1$ or $3$.
The first endpoint is smaller.  Indeed the difference is
\[
 J(1)-J(3)=3F(4C/3)-2F(C)-F(3C)<0,
\]
as follows from the integral expression for $F$ and
\[
 (C-u)^2(3C-u)-(4C/3-u)^3
 =C\{17C^2/27-5Cu/3+u^2\}>0
\]
for $C\ge2$, $0\le u\le1$.  The expression in braces
increases in $C$, decreases in $u$, and is at least $5/27$.
It follows that
\[
 J\ge G(C):=F(C/2)+F(4C)-2F(C).
\]
By \eqref{tri:F-series},
$G(C)=\log2-\sum_{j\ge1}(2^j+4^{-j}-2)/(j(j+1)C^j)$
increases with $C>2$.  Thus $J>1/2$ by
\eqref{tri:constant-table}.

Apply the same binomial estimates as in
Lemma~\ref{tri:dense-balanced}.  They give
\[
 \log\frac{B(c,b)\binom ab\binom db}{B(\ell,3b)}
 \ge bJ-\log b-
 \{\log(2\pi)-\tfrac12\log3+F'(\ell/(6b))\}
       -\frac1{4b}-\frac1{12\lceil\ell/2\rceil}
 >\frac b2-\log b-3-\frac1b>0.
\]
For the last step, use $\pi<22/7$, $\ell/(6b)\ge17/5$,
and $b\ge35$.  The parity of $\ell$ is already included in
these estimates.
\end{proof}

\subsection{Completion of the reduction}

\begin{proposition}\label{tri:long-range}
If $b\ge35$, $\mathcal C(b)$ holds, and
$S\in\Fcal(C_{3b})$ satisfies $|S|\ge102b/5$, then
\[
 \mathsf N_0^{3b}(S)\ge B(|S|,3b).
\]
\end{proposition}
\begin{proof}
Let $a\ge d\ge c$ be the three residue-class sizes.
If $c=0$, use Lemma~\ref{tri:two-class}.
If $a\le d+c$, use Lemma~\ref{tri:thin-balanced} when
$c\le d/4$ and Lemma~\ref{tri:dense-balanced} when $c\ge d/4$.
If $a\ge16|S|/25$, use Lemma~\ref{tri:large-class}.
The remaining case satisfies
$d+c<a<16(d+c)/9$.
If $c\le d/4$, the alternatives preceding
Lemma~\ref{tri:one-sided-endpoint} are exhaustive:
Lemma~\ref{tri:cap-failure} treats capacity failure,
Lemma~\ref{tri:capped} treats a feasible inverse with
$\alpha_b\ge\kappa^bR$, and Lemma~\ref{tri:low-alpha}
treats $\alpha_b<\kappa^bR$.
If $c\ge d/4$, use Lemma~\ref{tri:dense-one-sided} when
$a\ge4c$ and Lemma~\ref{tri:comparable-majority} when $a\le4c$.
This covers every profile.
\end{proof}

We now remove the restriction $b\ge35$ imposed at the beginning
of the section.

\begin{lemma}[Tripling reduction]\label{lem:tripling}
For every $b\in\N$, $\mathcal C(b)$ implies $\mathcal C(3b)$.
\end{lemma}
\begin{proof}
For $b\ge60$, Lemma~\ref{lem:short-range} proves the assertion
for $|S|\le\lfloor102b/5\rfloor$, and
Proposition~\ref{tri:long-range} proves it for
$|S|\ge\lceil102b/5\rceil$.  These ranges cover every integer
length.  The same argument applies to $3b=105$ and $165$, the
two exceptional moduli included in Lemma~\ref{lem:short-range}.

For the other multiples of $3$ below $180$, remove all powers
of $2$.  The resulting odd integer is either a prime power or
of the form $p^\alpha q$, with $p,q$ distinct primes,
unless it is $105$ or $165$.  Indeed three distinct odd prime
divisors require at least $3\cdot5\cdot7=105$, and the only
other such multiple of $3$ below $180$ is $3\cdot5\cdot11=165$;
two distinct odd prime divisors occurring twice require at
least $3^2\cdot5^2=225$.  Theorem~\ref{thm:kisin} handles
the stated forms, and Lemma~\ref{lem:doubling} restores the
removed powers of $2$.  The proof is complete.
\end{proof}

\section{Proof of Theorem~\ref{thm:main}}
\label{sec:pairing}

The reductions of the preceding sections apply when a quotient
sequence has a sufficiently frequent term. We now consider the
complementary case. The argument combines a simultaneous choice of
pairs with the representation bound in
Lemma~\ref{lem:representations}.

\begin{proposition}
\label{prop:pairing}
Let $n\ge385$ be an integer all of whose prime divisors are at
least $5$, and let $S\in\mathcal F(C_n)$ have length
$\ell\ge34n/5$. Suppose that
\[
 \mathsf h(S)\le\lfloor\ell/2\rfloor
 \quad\text{and}\quad
 \mathsf h(\pi_q(S))\le
       \left(\frac12+\frac{41}{100q}\right)\ell
 \quad\text{for every prime }q\mid n.
\]
Then $\mathsf N_0^n(S)\ge B(\ell,n)$.
\end{proposition}

\begin{proof}
Write $\ell=2N+\varepsilon$, where $\varepsilon\in\{0,1\}$.
If $\varepsilon=1$, reserve one position of $S$ as a singleton.
We partition the remaining positions into $N$ pairs. Let $s$ be
the number of distinct prime divisors of $n$. We shall use
\begin{equation}
 s\le n/5\qquad\text{and}\qquad N\ge3n.
 \label{eq:pair-size}
\end{equation}
Indeed, $n\ge5^s\ge5s$, and the second inequality follows from
$\ell\ge34n/5$ and $n\ge385$.

\smallskip
\noindent\emph{Choice of pairs.}
Start with a uniform perfect matching of the $2N$ positions.
For each prime $q\mid n$, let $I_q$ be the number of pairs whose
terms have the same image under $\pi_q$. If the multiplicities in
the remaining quotient sequence are $c_1,\ldots,c_q$ and
$c=\max c_j$, then
\[
 \mathbb E I_q
 =\frac{\sum_{j=1}^q\binom{c_j}{2}}{2N-1}
 \le\frac{N(c-1)}{2N-1}\le\frac c2<\frac{59N}{100}.
\]
The strict inequality follows from
$c\le(291/500)\ell$ and
$(291/500)(1+1/(2N))<59/100$.

Expose the pairs successively, choosing the partner of the first
unmatched position at each step. For a fixed $q$, consider the
conditional expectation of $I_q$. If $x$ and $y$ are two possible
partners of the next position $a$, the switch
\[
 (a,x),(y,z)\longmapsto(a,y),(x,z)
\]
bijects the two completion spaces and changes $I_q$ by at most two.
Hence each conditional range of the exposure martingale has length
at most two. Hoeffding's bounded-interval estimate \cite{Hoeffding1963}, applied
conditionally, gives a moment bound $\exp(t^2/2)$ for each centered
increment. Multiplication over at most $N$ increments and Markov's
inequality therefore give
\[
 \mathbb P(I_q-\mathbb E I_q\ge u)
       \le\exp\left(-\frac{u^2}{2N}\right).
\]
Consequently
\begin{equation}
 \mathbb P\bigl(I_q>7N/10\text{ for some prime }q\mid n\bigr)
 \le s e^{-121N/20000}
 \le\frac n5e^{-363n/20000}<1.
 \label{eq:matching-exists}
\end{equation}
For the last inequality, the logarithm of the middle expression is
decreasing for $n\ge385$. At $n=385$ its negative exponent exceeds
$6$, whereas $n/5=77$ and
$e^6>1+6+18+36+54=115$.

Fix a matching for which $I_q\le7N/10$ for every prime $q\mid n$.
We next remove pairs with equal terms without increasing any $I_q$.
If $(g,g)$ is one such pair, some other pair $(x,y)$ has both terms
different from $g$. Otherwise the $N$ pairs would contain at least
$N+1$ copies of $g$, contrary to $\mathsf h(S)\le N$.
Replace these pairs by $(g,x)$ and $(g,y)$. No new equal pair is
created, and at least one is removed.

For any quotient of $C_n$, the number of internal pairs does not
increase under this switch. If both, exactly one, or neither of
$x,y$ have the image of $g$, the internal-pair counts change as
$2\to2$, $1\to1$, or $1\text{ or }2\to0$, respectively.
Repeating the switch gives a fixed matching with differences
$d_1,\ldots,d_N$ satisfying
\begin{equation}
 d_i\ne0\quad(i\in[1,N]),\qquad
 |\{i:d_i\in qC_n\}|\le7N/10
       \quad(q\mid n\text{ prime}).
 \label{eq:fixed-matching}
\end{equation}
Only the existence of this matching is used in what follows.

\smallskip
\noindent\emph{Selections whose transversal sums do not fill the group.}
Let $r\in\{n-1,n\}$, and select $r$ pairs uniformly from the fixed
matching. After relabeling their differences, the set of transversal
sums is a translate of
\[
 T=\{0,d_1\}+\cdots+\{0,d_r\}.
\]
A transversal chooses exactly one position from each selected pair.
Suppose that $T\ne C_n$, and let $H$ be its stabilizer. If $H=\{0\}$,
Kneser's theorem gives $|T|\ge r+1\ge n$, a contradiction.
Thus $H$ is a proper nontrivial subgroup. Put $d=(C_n:H)$ and
let $u$ count selected differences outside $H$. Since $|H|\ge5$,
we have $d\le n/5$. Applying Kneser's theorem in the quotient yields
$|T/H|\ge u+1$. Properness gives $u\le d-2$, so more than $4r/5$
selected differences lie in $H$:
\[
 r-u\ge r-d+2\ge r-n/5+2>4r/5.
\]
Choose a maximal subgroup containing $H$. It is $qC_n$ for some
prime $q\mid n$.

Let $\mathcal B$ be the event that there exists a prime $q\mid n$
such that more than $4r/5$ selected differences belong to $qC_n$.
We have proved that every selection outside $\mathcal B$ has all
elements of $C_n$ as transversal sums.

\smallskip
\noindent\emph{A probability bound proportional to the nongenerator density.}
Let $\rho$ be the fraction of the $N$ differences in
\eqref{eq:fixed-matching} which do not generate $C_n$, and put
\[
 \theta_q=\frac{|\{i:d_i\in qC_n\}|}{N}.
\]
Then $\theta_q\le\rho$ and $\theta_q\le7/10$.
Let $X_q$ be the number of selected differences in $qC_n$.
For $z>1$, its moment generating polynomial satisfies
\begin{equation}
 \mathbb E z^{X_q}\le(1-\theta_q+\theta_q z)^r.
 \label{eq:hypergeometric}
\end{equation}
Indeed, the left side is the $r$th elementary symmetric mean of
$N$ numbers, with $N\theta_q$ equal to $z$ and the others equal
to $1$. The symmetric-mean inequality bounds this by the $r$th
power of their arithmetic mean. This inequality also follows by
repeatedly averaging two variables: with their sum fixed, the
elementary symmetric polynomial has the form $A+B(x+y)+Cxy$
with $C\ge0$, and therefore does not decrease upon averaging.

For $0<\theta_q<4/5$, optimization in Markov's inequality gives
\[
 \mathbb P(X_q>4r/5)
 \le e^{-rD(4/5\,\Vert\,\theta_q)},
 \quad
 D(x\,\Vert\,v)=x\log\frac xv+(1-x)\log\frac{1-x}{1-v}.
\]
For $0<v\le7/10$ and $r\ge n-1\ge384$,
\[
 \frac{d}{dv}\log\left(\frac{e^{-rD(4/5\,\Vert\,v)}}v\right)
 =\frac{4r/5-1-(r-1)v}{v(1-v)}>0.
\]
Moreover $D(x\,\Vert\,v)\ge2(x-v)^2$, since as a function of
$x$ it has value and derivative zero at $x=v$ and second derivative
$1/[x(1-x)]\ge4$. Hence
$D(4/5\,\Vert\,7/10)\ge1/50$, and
\[
 \mathbb P(X_q>4r/5)\le\frac{10}{7}\theta_q e^{-r/50}.
\]
The probability is zero if $\theta_q=0$. Summing over the prime
divisors and using \eqref{eq:pair-size}, we obtain, for $\rho>0$,
\begin{equation}
 \mathbb P(\mathcal B)
 \le\frac{10s}{7}\rho e^{-r/50}
 \le\frac{2n}{7}\rho e^{-(n-1)/50}<\frac\rho3.
 \label{eq:proportional-failure}
\end{equation}
To verify the last inequality, note that
$\log(6n/7)-(n-1)/50$ decreases for $n\ge385$.
At $n=385$, $(n-1)/50>7$, while the exponential series through
degree six gives $e^7>330=6n/7$. If $\rho=0$, then
$\mathcal B$ is empty.

\smallskip
\noindent\emph{Compensation and the final count.}
Let $\mathcal A$ be the event that the selected pairs include a
difference which does not generate $C_n$. Then
$\mathbb P(\mathcal A)\ge\rho$, since the first pair in a uniform
ordered sample already has this probability.

Put $k=2$ if $r=n$ and $k=1$ if $r=n-1$. For each selection outside
$\mathcal B$, every target is attained. By
Lemma~\ref{lem:representations}, there are at least $k$ transversals
with any prescribed sum, and at least $k+1$ on $\mathcal A$.
The target may depend on the selected pairs: subtracting the sum
of one fixed endpoint in each pair identifies transversal sums
with subset sums of the differences. Thus the expected number of
transversals with the required sum is at least
\begin{equation}
 k\mathbb P(\mathcal B^c)+\mathbb P(\mathcal A\cap\mathcal B^c)
 \ge k+\mathbb P(\mathcal A)-(k+1)\mathbb P(\mathcal B)
 \ge k.
 \label{eq:surplus-compensation}
\end{equation}
The last conclusion is immediate as well when $\rho=0$.

For $r=n$, prescribe sum zero. If $\varepsilon=1$, also use
$r=n-1$, prescribing the negative of the reserved term, and adjoin
that term. A positional transversal determines its selected pairs,
so different selections do not give the same positional subsequence.
The subsequences using the singleton are disjoint from those
omitting it. Summing \eqref{eq:surplus-compensation} therefore gives
\[
 \mathsf N_0^n(S)\ge
 2\binom Nn+\varepsilon\binom N{n-1}
 =\binom Nn+\binom{N+\varepsilon}n
 =B(\ell,n).
\]
This proves the proposition.
\end{proof}

\bigskip

Now we are in a position to prove Theorem~\ref{thm:main}.

\begin{proof}[Proof of Theorem~\ref{thm:main}]
We prove $\mathcal C(n)$ by strong induction on $n$. The assertion
for $n=1$ is immediate. We may use Theorem~\ref{thm:kisin} whenever
$n$ is a prime power or has the form $p^a q$.
If $n$ is even, apply Lemma~\ref{lem:doubling} and the induction
hypothesis at $n/2$. If $n$ is odd and divisible by $3$, apply
Lemma~\ref{lem:tripling} and the induction hypothesis at $n/3$.

It remains to consider an integer $n$ outside Kisin's cases all of
whose prime divisors are at least $5$. If it has at least three
distinct prime divisors, then $n\ge5\cdot7\cdot11=385$. If it has
exactly two, both exponents are at least two, and
$n\ge5^2 7^2=1225$. Thus $n\ge385$ in every remaining case.
For each prime $q\mid n$ we also have $n/q\ge25$: an integer below
$25$ with no prime divisor below $5$ is $1$ or a prime, which would
place $n$ in one of Kisin's cases.

Let $S\in\mathcal F(C_n)$ and put $\ell=|S|$. If
$\ell\le\lfloor34n/5\rfloor$, the assertion follows from
Lemma~\ref{lem:short-range}. Hence we may suppose
$\ell\ge34n/5$. If $\mathsf h(S)\ge\lfloor\ell/2\rfloor$, apply
Lemma~\ref{lem:majority}.

Suppose next that some prime $q\mid n$ satisfies
\[
 \mathsf h(\pi_q(S))>
 \left(\frac12+\frac{41}{100q}\right)\ell.
\]
Set $b=n/q$. Then $q\ge5$, $b\ge25$, and $\mathcal C(b)$ holds
by induction. Lemma~\ref{lem:large-class} yields the required bound.
Otherwise all hypotheses of Proposition~\ref{prop:pairing} hold,
and that proposition gives $\mathsf N_0^n(S)\ge B(\ell,n)$.
The completes the proof.
\end{proof}

\section*{Acknowledgements}

This work was supported by the National Natural Science Foundation
of China (grant nos. 12371335 and 12271520).

\section*{Declaration of generative AI and AI-assisted technologies}

The authors used a large language model during the preparation of this manuscript
for assistance with mathematical exploration, checking arguments,
organization and language editing. The authors are responsible for
the mathematical statements, proofs, references and final text.

\end{document}